\documentclass[a4paper]{article}

\usepackage[a4paper,top=3cm,bottom=2cm,left=3cm,right=3cm,marginparwidth=1.75cm]{geometry}
\usepackage{amsmath, amsthm, amsfonts}
\usepackage{enumerate}
\usepackage{hyperref}
\usepackage{amssymb}
\usepackage{graphicx}
\usepackage[all]{xy}
\newtheorem{thm}{Theorem}[section]
\newtheorem{cor}[thm]{Corollary}
\newtheorem{lem}[thm]{Lemma}
\newtheorem{prop}[thm]{Proposition}

\theoremstyle{definition}
\newtheorem{defi}[thm]{Definition}
\theoremstyle{remark}
\newtheorem{nota}[thm]{Notation}
\newtheorem{rem}[thm]{Remark}

\providecommand{\keywords}[1]
{

	\small	
	\textbf{\textit{Keywords:}} #1
}

\providecommand{\AMSMSC}[1]
{

	\small	
	\textbf{\textit{AMS MSC:}} #1
}

\begin{document}
%-----------------------------------------------------------------
\title{Ideals of some Green biset functors }
\author{J. Miguel Calderón \\
calderonl@matmor.unam}
\maketitle
\begin{center}
Centro de Ciencias Matemáticas,\\
 Universidad Nacional Autónoma de México, \\ 
 Antigua Carretera a Pátzcuaro 8701, 58089 Morelia, Michoacán, Mexico

\end{center}
\begin{abstract}
%% Text of abstract
In this article, we describe the lattice of ideals of some Green biset functors. We consider Green biset functors which satisfy that each evaluation is a finite dimensional split semisimple commutative algebra and use the idempotents in these evaluations  to characterize any ideal of  these Green biset functors.  For this we will give the definition of $MC$-group, this definition  generalizes  that of a $B$-group, given for   the Burnside functor.\\
Given a Green biset functor $A$, with the above hypotheses, the set of all $MC$-groups of $A$ has a    structure of a poset  and we  prove that  there exists an isomorphism
 of lattices between the set of ideals of $A$ and  the set  of upward closed subsets of the $MC$-groups of $A$.
\end{abstract}
\AMSMSC{16Y99, 18D99, 20J15.}

\section{Introduction}
In \cite{serge-biset} Serge Bouc introduced and developed the theory of biset functors. This notion provides a framework for situations where structural maps that behave like
restriction, induction, inflation, and deflation, or a subset of them, are present. Typical
examples of biset functors are the various representation rings, for instance the Burnside
ring, the character ring, the Green ring, and the trivial source ring of a finite group $G$.\\
A Green functor for a finite group $G$ is defined as a Mackey functor $A$ with an additional multiplicative structure on each $A(H)$, for $H$ a subgroup of $G$, compatible with the structure of Mackey
functor. In the context of Mackey functors, Green functors have been extensively studied and many
examples and applications of them have been found (see for example Thévenaz \cite{thevenaz1} and Bouc \cite{boucgreen}). In the context of categories of biset functors, it has been proved by
Serge Bouc that if the class of objects in the biset category is closed under direct products, then the
category of biset functors defined on it has a symmetric monoidal structure given by the tensor product and 
the identity element is the Burnside functor. A Green biset functor $A$ is then a monoid in this category. That is, $A$ is a biset functor compatible with the monoidal structures of the biset category and
that of $R$-Mod, when $R$ is a commutative ring with unity. This means that $A$ is equipped with bilinear products from $A(G) \times A(H)$ to $A(G \times H)$, for finite groups $G$ and $H$, which have a unit element
and are associative and functorial in a natural sense. One feature of this practical definition is that it
allows us to observe that many known Green functors are Green biset functors too.  In \cite{serge-biset}  Bouc studied  the  Burnside functor with coefficients in  $\mathbb{Q}$, denoted by $\mathbb{Q}B$.  This functor is a Green  biset functor such that in each evaluation  is  a finite dimensional split semisimple commutative  $\mathbb{Q}$-algebra. Bouc  gave a characterization of ideals of $\mathbb{Q} B$.
The aim of this article is  to generalize the techniques used in \cite{serge-biset} by  Bouc, to find a  characterization
 of  the ideals of any  Green biset functor with similar characteristics to $\mathbb{Q}B$.  This means that the functors we consider  are   finite dimensional split semisimple commutative algebra for any evaluation.  If $A$ is a Green biset functor  with the above hypotheses, in Section \ref{biset and ideals}, we study the effect  of bisets in any primitive idempotent of the evaluations and  we will give the definition of $MC$-group. This definition is a generalization that of a $B$-group, given for  the Burnside functor in   \cite{serge-biset}.  An   ideal of $A$ is a submodule of the $A$-module $A$. In section \ref{ideals and idem}, we will study the  relationship  between any  ideal   and the ideals generated by the  primitive idempotents of any evaluation, with  this  we give an isomorphism of lattices between the lattice of ideals of $A$ and the lattice of upward closed subsets of the $MC$-groups of $A$. In the last section \ref{examples},  we compare our  characterization of  the ideals  with some  classic examples of Green biset functor.

\section{Preliminaries}
Throughout the paper, we fix a commutative unital ring $R$. All referred groups will be
finite.
\subsection{Green Biset functors}
The biset category $R\mathcal{C}$ over $R$ has  all finite groups as objects, and for finite groups $G$ and $H$,  the hom-set  $Hom_{R\mathcal{C}}(G, H)$ is $RB(H, G) := R \otimes_\mathbb{Z}  B(H, G)$, where $B(H, G)$ is the Grothendieck group of the category of finite
$(H, G)$-bisets. The composition of morphisms in $R\mathcal{C}$ is induced by the usual tensor product of bisets,  which will be denoted by $\circ$.\\
We fix a non-empty class $\mathcal{D}$ of finite groups closed under subquotients and cartesian
products. We
denote by $R\mathcal{D}$ the full subcategory of $R\mathcal{C}$ consisting of groups in $\mathcal{D}$, so in particular
$R\mathcal{D}$ is a replete subcategory of $R\mathcal{C}$ (Definition 4.1.7. of \cite{serge-biset}). A biset functor over $R\mathcal{D}$ or  $R\mathcal{D}$-biset functor is an $R$-linear functor from $R\mathcal{D}$ to the category $R$-Mod of $R$-modules. Biset functors over $R\mathcal{D}$ form an abelian category,  where morphisms are natural transformations of functors,  it will be denoted by $\mathcal{F}_{D,R}$.\\
A Green $\mathcal{D}$-biset functor is defined as a monoid in $\mathcal{F}_{D,R}$ (see Definition 8.5.1
in \cite{serge-biset}). This is equivalent to the following definition:
\begin{defi} \label{Defi1}
An  $R\mathcal{D}$-biset functor $A$ is a Green biset functor over $R\mathcal{D}$ if it is equipped with
bilinear products $A(G) \times A(H) \longrightarrow A(G \times H$) denoted by $(a, b) \longmapsto a \times b$, for groups $G$, $H$ in $\mathcal{D}$, and an element $\xi_A \in A(1)$, satisfying the following conditions:
\begin{itemize}
\item Associativity. Let $G$, $H$ and $K$ be groups in $\mathcal{D}$. If we consider the canonical
isomorphism  $\phi: G \times (H \times K) \longrightarrow (G \times H) \times K$, then for any $a \in A(G)$,
$b\in A(H)$ and $c \in A(K)$
\begin{align*}
(a \times b) \times c = A(Iso \phi)(a\times (b\times c)).
\end{align*}
\item Identity element. Let $G$ be a group in $\mathcal{D}$ and consider the canonical isomorphisms
$\lambda_G :1 \times G \longrightarrow G $ and $\overline{\lambda}_G: G \times 1 \longrightarrow G$. Then for any $a \in A(G)$
\begin{align*}
a=A(Iso \lambda_G)(\xi_A \times a)=A(Iso \overline{\lambda}_G )(a\times \xi_A).
\end{align*}
\item Functoriality. If $\phi : G \longrightarrow  G^\prime$ and $\psi : H \longrightarrow H^\prime$ are morphisms in $R\mathcal{D}$, then for any $a\in A(G)$ and $b \in A(H)$
\begin{align*}
A(\phi \times \psi)(a \times b) = A(\phi)(a) \times A(\psi)(b).
\end{align*}

\end{itemize}
\end{defi}

If $A$ and $C$ are Green $R\mathcal{D}$-biset functors, a morphism of Green $R\mathcal{D}$-biset functors from $A$ to $C$ is a natural transformation $f : A \longrightarrow C$ such that $f_{H\times K} (a \times b) = f_H (a) \times f_K(b)$ for any groups $H$ and $K$ in $\mathcal{D}$ and any $a \in A(H)$, $b\in  A(K)$, and such that $f_1(\xi_A) = \xi_C$.\\

In Section 8.5 of  \cite{serge-biset} it is shown that this definition is equivalent to the
following definition,  as we see in the next lemma.

\begin{defi} \label{green}
 An $R\mathcal{D}$-biset functor  $A$ is a Green biset functor  over $R\mathcal{D}$ if  for each
group $H$ in $\mathcal{D}$, the $R$-module $A(H)$ is an $R$-algebra with unity that satisfies the following.  Let  $K$ and $G$ be  groups in $\mathcal{D}$ and let $\phi:K \longrightarrow G$  be a group homomorphism, then:
\begin{itemize}
\item  For the $(K, G)$-biset $G$,  denoted by $_{K^\phi}G_G$, the morphism $A(_{K^\phi}G_G)$ is a ring homomorphism.
\item For the ($G, K)$-biset $G$, denoted by $_G G_{^\phi K}$, the morphism $A(_G G_{^\phi K})$ satisfies the Frobenius identities for all $b \in A(G)$ and $a \in A(K)$,
\begin{align*}
A(_G G_{^\phi K})(a) \cdot b&= A(_G G_{^\phi K})(a\cdot A(_{K^\phi} G_G)(b))\\
b \cdot A(_G G_{^\phi K})(a)  &= A(_GG_{^\phi K})(A(_{K^\phi} G_G )(b) \cdot a).
\end{align*}
where $\cdot$ denotes the ring product on $A(G)$, respectively  $A(K)$.
\end{itemize}
\end{defi}

\begin{lem}[Lemma 3 in \cite{center}]\label{equi de defi}
 The two previous definitions are equivalent. Starting
with Definition \ref{Defi1}, the ring structure of $A(H)$ is given by
\begin{align*}
a\cdot b = A\left( Iso^H_{\Delta(H)} \circ Res^{H\times H}_{\Delta(H)}\right)  (a\times b),
\end{align*}
for $a$ and $b$ in $A(H)$, with the unity given by $A(Inf_1^H)(\xi_A)$. Conversely, starting with
Definition \ref{green}, the product of $A(G) \times  A(H) \longrightarrow  A(G \times H)$ is given by
\begin{align*}
a\times b = A (Inf^{G\times H}_G)(a)\cdot A(Inf^{G\times H}_H)(b),
\end{align*}
for $a\in A(G)$ and $b\in A(H)$, with the identity element given by the unity of $A(1)$.
\end{lem}

\subsection{$A$-modules}

\begin{defi}[Definition 8.5.5 in \cite{serge-biset}]
 Given a Green biset functor $A$, a left $A$-module
$M$ is defined as a biset functor, together with bilinear products
$$\times : A(G) \times M(H) \longrightarrow  M(G \times H)$$
for every pair of groups $G$ and $H$ in $\mathcal{D}$, that satisfy analogous conditions to those of
Definition $\ref{Defi1}$. The notion of right $A$-module is defined similarly, with bilinear products
$M(G) \times A(H) \longrightarrow M(G \times H)$.
\end{defi}
If $M$ and $N$ are $A$-modules, a morphism of $A$-modules is defined as a morphism
of biset functors $f : M \longrightarrow N$ such that $f_{G\times H} (a \times m) = a \times  f_H(m)$ for all groups $G$ and $H$ in $\mathcal{D}$, $a \in  A(G)$ and $m \in M(H)$.  With these morphisms, the $A$-modules form an abelian category, denoted by $A$-Mod.
\begin{defi}
 Let $A$ be a Green biset functor over $R\mathcal{D}$. A left ideal of $A$ is an $A$-submodule of the left $A$-module $A$. In other
words,  it is a biset subfunctor $I$ of $A$ such that the image of 
\begin{align*}
A(G) \times I(H)\longrightarrow A(G \times  H)
\end{align*}
is contained in $I(G\times H)$ for any objects $G$ and $H$ of $\mathcal{D}$.
One defines similarly a right ideal of $A$.  This is a two-sided ideal of $A$ is a left
ideal which is also a right ideal, if  $I$ is a two-sided of $A$, we denote this by $I\leq A$.
\end{defi}
 The  left ideal generated by $x \in A(G)$ is by definition the   intersection of all left ideals $I$ of $A$ such that $x\in I(G)$. Similarly, one defines a right ideal generated by $x$, and then a two-sided ideal generated by $x$.\\
From Proposition 8.6.1 of \cite{serge-biset}, or Proposition 2.11 of \cite{simpleRomero}, an equivalent way of defining an $A$-module is as an $R$-linear functor from the category $\mathcal{P}_A$ to $R$-Mod, the
category $\mathcal{P}_A$ being defined next.
\begin{defi} \label{P_A}
Let $A$ be an $R\mathcal{D}$-Green biset  functor over $R$. The category $\mathcal{P}_A$ is defined in the
following way:
\begin{itemize}
\item The objects of $\mathcal{P}_A$ are all finite groups in $\mathcal{D}$.
\item  If $G$ and $H$ are groups in $\mathcal{D}$, then $Hom_{\mathcal{P}_A}
(H, G) := A(G \times H)$.
\item Let $H$, $G$ and $K$ be groups in $\mathcal{D}$. The composition of $\beta \in  Hom_{\mathcal{P}_A}
(H, G) $ and
$\alpha \in Hom_{\mathcal{P}_A}
(G, K)$ in $\mathcal{P}_A$ is the following:
\begin{align*}
\beta \circ \alpha = A(Def^{H\times \Delta(G) \times K}_{H\times K} \circ Res^{H\times G \times G \times K}_{H\times \Delta(G) \times K}) (\beta \times \alpha).
\end{align*}
\item For a group $G$ in $\mathcal{D}$, the identity morphism $1_G$ of $G$ in $\mathcal{P}_A$ is $ A(Ind^{G\times G}_{\Delta(G)} \circ Inf^{\Delta(G)}_1)(\xi_A)$.
\end{itemize}

\end{defi}

Let $F$  be a functor from $\mathcal{P}_A$ to $R$-Mod and $I$  be a subfunctor of $F$, this is donote by $I\subseteq  F$.

\begin{rem} \label{inter}
Let $F$ be an $A$-module  . If $(I_j)_{j\in J}$ is a set of $A$-submodules of $F$, then the
intersection $ \cap_{j\in J} I_j$ is an $A$-submodule of $F$  whose evaluation at the object $G$
of $R\mathcal{D}$ is equal to
 \begin{align*}
 \left( \cap_{j\in J} I_j  \right) (G) =  \cap_{j\in J} I_j (G) 
 \end{align*}
In particular, let $G$ be  an of object of $\mathcal{D}$ and  $\Gamma_G$
be a subset of $F(G)$. The submodule  $F_{\Gamma_G}$ of $F$ generated by $\Gamma_G$ is
by definition the intersection of all $A$-submodules  $F^\prime$ of $F$ such that $\Gamma_G    \subset  F^\prime(G)$.  If $H$ is an object of $\mathcal{D}$, it is easy to see that
\begin{align*}
F_{\Gamma_G}  (H) = \sum_{\gamma \in \Gamma_G}  F\left(  Hom_{R\mathcal{D}}(G, H)\right) (\gamma).  
\end{align*}
\end{rem}

\begin{proof}
Let  $T$  be the $A$-submodule  of $F$ such that  in an  object $H$ of  $\mathcal{D}$ is define by 
\begin{align*}
T (H) = \sum_{\gamma \in \Gamma_G}  F\left(  Hom_{R\mathcal{D}}(G, H)\right) (\gamma) . 
\end{align*}
By definition of  $F_{\Gamma_G}$, one has $F_{\Gamma_G}$ is an $A$-submodule of $T$. \\
 On the other hand, if $F^\prime$ is an  $A$-submodule   of $F$ such that $\Gamma_G    \subset  F^\prime(G)$, then for any $H$ object of $\mathcal{D}$,  $\alpha \in   Hom_{R\mathcal{D}}(G, H)$ and any $\gamma \in \Gamma_G$  one has $ F^\prime(\alpha)(\gamma) \in F(H)$, then $T$ is an $A$-submodule of $F_{\Gamma_G}$.
\end{proof}

\begin{lem} \label{a comp b}
If $\alpha \in A(H\times K ) $ and $\beta \in A(K\times G)$, then $\alpha\circ \beta$ is equal to 
\begin{align*}
A(Def^{H\times K\times G}_{H\times G})\left( A(Inf_{H\times K}^{H\times K\times G} )(\alpha)\cdot A(Inf^{H\times K\times G}_{K\times G})(\beta)\right).
\end{align*}
\end{lem}

\begin{proof}
By lemma \ref{equi de defi}, one has 
\begin{align*}
\alpha\times \beta = A (Inf^{H\times K\times K\times G}_{H\times K})(\alpha )\cdot A(Inf^{H\times K\times K \times G}_{K\times G})(\beta).
\end{align*}
By the Relations 1.1.3 of \cite{serge-biset}, one has the following bisets isomorphism
\begin{align*}
Res_{H\times \Delta(K)\times G}^{H\times K\times K \times G} \circ Inf^{H\times K\times K \times G}_{H\times K} \cong Inf^{H\times \Delta(K)\times G}_{H\times \Delta(K)} \circ (H\times Iso(\varphi)) 
\end{align*}
and
\begin{align*}
Res_{H\times \Delta(K)\times G}^{H\times K\times K \times G} \circ Inf^{H\times K\times K \times G}_{K\times G} \cong Inf^{H\times \Delta(K)\times G}_{ \Delta(K)\times G} \circ (Iso(\varphi)\times G) 
\end{align*}
where $\varphi : \Delta(K) \longrightarrow  K$ is the natural isomorphism. On the other hand, 
\begin{align*}
Inf^{H\times \Delta(K)\times G}_{H\times \Delta(K)} \circ (H\times Iso(\varphi)) \cong (H\times Iso(\varphi) \times G) \circ Inf^{H\times K \times G}_{H\times K}
\end{align*}
and
\begin{align*}
Inf^{H\times \Delta(K)\times G}_{ \Delta(K)\times G} \circ (Iso(\varphi)\times G) \cong (H\times Iso(\varphi) \times G) \circ Inf^{H\times K \times G}_{K\times G}.
\end{align*}
Then  $\alpha \circ \beta$ is equal to $A(Def^{H\times \Delta(K)\times G}_{H\times K})$ applied to
\begin{align*}
A(H\times Iso(\varphi)\times G)\left( A(Inf^{H\times K\times G}_{H\times K})(\alpha)\cdot A(Inf^{H\times K \times G}_{K\times G})(\beta)\right).
\end{align*}
We can note that 
\begin{align*}
( Def^{H\times \Delta(K)\times G}_{H\times K} ) \circ (H\times Iso(\varphi)\times G) =  Def^{H\times  K \times G}_{H\times K}. 
\end{align*}
Thus 
\begin{align*}
\alpha \circ \beta= A(Def^{H\times K\times G}_{H\times G})\left( A(Inf_{H\times K}^{H\times K\times G} )(\alpha)\cdot A(Inf^{H\times K\times G}_{K\times G})(\beta)\right).
\end{align*}

\end{proof}
\subsection{Examples of Green biset functors}
In this subsection,  we revise some classical examples of Green biset functors.  We will study  these functors in  Section \ref{examples}.
\subsubsection{The Burnside functor RB.}
The Burnside functor on $R\mathcal{D}$ is  defined as
\begin{align*}
RB = Hom_{R\mathcal{D}}(1,-).
\end{align*}
In other words, $RB$ is the Yoneda functor corresponding to the trivial
group, which is an object of $\mathcal{D}$ since the class of objects of $\mathcal{D}$ is closed under
 quotients. Thus, for an object $G$ of $\mathcal{D}$,  the $R$-module $RB(G)$ is equal to
$R \otimes_\mathbb{Z} B(G)$. If $H$ is another object of $\mathcal{D}$, and if $U$ is a finite $(H, G)$-biset,
then the map $RB(U) : RB(G) \longrightarrow RB(H)$ is induced by the correspondence
sending a finite $G$-set $X$ to the $H$-set $U \circ_G X$.  The cross product of sets defines the bilinear products
$$RB(G) \times  RB(H) \longrightarrow RB(G \times  H)$$ that make $RB$ a Green biset functor.
 
 \subsubsection{Fibered Burnside Functor} \label{sect fibered functor}
The second example of a Green biset functor that we will see is the fibered Burnside functor (see\cite{fibered} or \cite{romerofibered}). We fix a multiplicative Group group $A$. For every finite group
 $G$, we set 
 \begin{align*}
 G^\ast := Hom (G, A).
 \end{align*}
Let  $X$ be a set, we call $X$ an $A$-fibered $G$-set if $X$ is an $A \times G$-set such that the action of $A$ is free with finitely many orbits. We
denote by $_Gset^A$ the category of $A$-fibered $G$-sets. Here the morphisms are given by $A \times G$-equivariant functions. The operation of disjoint union of sets induces a coproduct on $_Gset^A$
and we denote by $B^A (G)$ the Grothendieck group of this category with respect to disjoint
unions. The group $B^A (G)$ is called the $A$-fibered Burnside group, and it was first introduced,
in a more general way, by Dress in \cite{dressRing}.\\
The functor $B ^A$ is defined as follows. In objects, it sends a group $G$ to $B^A(G )$. In arrows, for
a $(G, H)$-biset $U$, the map $B^A (U) : B^A (H) \longrightarrow B^A (G)$ is the linear extension of the correspondence
$X \longrightarrow U \otimes_{AH} X$ , where $ X$ is an $A$-fibered  $H$-set and  $U \otimes_{AH} X$ is  the subset of elements of $U \times_H X$   in which $A$ acts freely.\\
Let $X$ be an $A$-fibered $G$-set and $Y$ be an $A$-fibered
$H$-set.  The set of $A$-orbits of $X \times Y$ with respect to the action $a(x, y) = (ax, a^{-1} y)$,  the orbit of $(x, y)$ is denoted by $x\otimes_A  y$. We extend this product by linearity and denote it by $X \otimes_A Y$ . The action of $A$ in  $x \otimes y$ is given by $ax \otimes  y$ and so it is easy to see that $A$ acts freely on $X\otimes_A  Y$.\\
The tensor product of $A$-fibered sets defines the bilinear products
\begin{align*}
B^A(G) \times  B^A(H) &\longrightarrow B^A(G \times  H)\\
(X,Y) &\longmapsto X\otimes_A Y
\end{align*}
that make $B^A$ a Green biset functor.

\subsubsection{Slice Burnside functor} \label{slice functor}
The third example of a Green biset functor that we will see is the slice Burnside functor (see \cite{boucslice} and \cite{slicefunctor}).

\begin{defi}
 The category $G$-Mor
of morphisms of $G$-sets has as objects the morphisms of $G$-sets, and a morphism from
$f :A \longrightarrow B$ to $g :A^\prime \longrightarrow  B^\prime$ is a pair of morphisms of $G$-sets $h :A \longrightarrow  A^\prime$ and $k : B \longrightarrow B ^\prime$
making the following diagram commute
\begin{equation}
\xymatrix{
A \ar[r]^f \ar[d]_h & B  \ar[d]_k \\
A^\prime \ar [r]^g & B^\prime.
}
\end{equation}

Note that the category $G$-Mor admits products (induced by the direct product of
$G$-sets) and coproduct (induced by the disjoint union of $G$-sets).
\end{defi}
\begin{defi} \label{xi(G)}
Let $G$ be a finite group. The slice Burnside group $\Xi(G)$ of $G$ is the
Grothendieck group of the category $G$-Mor, This is defined as the quotient
of the free abelian group on the set of isomorphism classes $[ X\xrightarrow{\;\; f \;\; }  Y]$ of morphisms of finite $G$-sets, by the subgroup generated by elements of the form
\begin{align*}
[X_1\sqcup X_2 \xrightarrow{\;\; f_1 \sqcup f_2 \;\; } Y] -[X_1\xrightarrow{\;\; f_1 \;\; } f_1(X_1) ] - [X_2 \xrightarrow{\;\; f_2 \;\; } f_2(X_2) ],
\end{align*}
whenever $X_1  \sqcup X_2  \xrightarrow{\;\; f_1 \sqcup f_2 \;\; } Y$ is a morphism of finite $G$-sets  where $ f_1 \sqcup f_2|_{X_1}= f_1 $ and $=f_1 \sqcup f_2 |_{X_2}=f_2 $ .\\
The  product of morphisms induces a commutative unital ring structure on $\Xi (G)$. The identity element for multiplication is the image of the class $[\bullet \longrightarrow \bullet]$, where $\bullet $ denote a $G$-set of cardinality $1$. For a morphism of $G$-sets $f:X \longrightarrow  Y$, let $\pi(f)$ denote the image in $\Xi(G)$ of the isomorphism class of $f$. 
\end{defi}

\begin{defi}
The functor $\Xi$ is defined as follows. In objects, it sends a group $G$ to $\Xi  (G )$. In arrows, for
a $(G, H)$-biset $U$, the map 
\begin{align*}
\Xi (U) : \Xi  (G) &\longrightarrow \Xi  (H) \\
(X \xrightarrow{\;\; f\;\; } Y) &\longmapsto (U\times_G X \xrightarrow{\;\; U\times_G f\;\; } U\times_G Y),
\end{align*}
 where  $U \times_G X$  and  $U \times_G Y$ have the natural action of $H $-sets coming from the action of $H$ on $U$.
\end{defi}
Let  $X \xrightarrow{\;\; f\;\; } Y$ a morphism of  $G$-set  and let  $ Z \xrightarrow{\;\; g\;\; } W$ a morphism of   $H$-set.  Define $X \times Z \xrightarrow{\;\; f\times g\;\; } Y \times W$  to  be  morphism of $G \times H$-set,  where  $X\times Z$ and  $Y\times W$   are $G \times H$-sets  in the natural way and the function  $X \times Z \xrightarrow{\;\; f\times g\;\; } Y \times W $ is define  by $f\times g(x,z):=\left(  f(x), g(z) \right) $,  we can extend these  definition to  the  bilinear product 
\begin{align*}
\Xi(G) \times  \Xi(H) &\longrightarrow \Xi (G \times  H)\\
(X \xrightarrow{\;\; f\;\; } Y, Z \xrightarrow{\;\; g\;\; } W) &\longmapsto X \times Z \xrightarrow{\;\; f\times g\;\; } Y \times W.
\end{align*}
This   bilinear products   make $\Xi$ a Green biset functor.

Now, we will give some definitions that we will use in  Section \ref{examples}.
\begin{defi}[Definition 3.1 of \cite{boucslice}] \label{slice}
Let $G$ be  a finite group. A  slice of $G$ is a pair  $(T,S)$ of subgroups of $G$, with $S\leq T$.
\end{defi}
The set of slices of $G$ is denoted by $\Pi (G)$. Given two slices  $(V,U)$ and $(T,S)$ of $G$, we say that $(V,U)$ is a quotient of $(T,S)$,   denoted by $(T,S) \twoheadrightarrow  (V,U)$, if there exists a surjective group homomorphism $\phi :T\longrightarrow V$ such that $\phi (S) =U.$  If $\phi$  is an isomorphism, we say that $(V,U)$ and $(T,S)$ are isomorphic. 

 If $(T, S)$ is a slice of $G$, set 
\begin{align*}
\langle T,S\rangle_G=\pi (G/S \longrightarrow G/T)
\end{align*}
as an element of $\Xi(G).$

\begin{lem} [\cite{boucslice}, Lemma 3.4] \label{base de slice} Let $f: X \longrightarrow Y$ be a morphism of $G$-sets, Then in the group $\Xi(G)$, 
\begin{align*}
\pi(f)=\sum_{x\in[ G\setminus X] } \langle G_{f(x)}, G_x \rangle,
\end{align*}
where $G_\bullet$ denotes the stabilizer of $\bullet$.
\end{lem}
Thus, the group $\Xi(G)$ is generated by the elements  $\langle T, S \rangle_G$, where   $  (T, S)$ runs through a set $[\Pi(G)]$ of representatives of conjugacy classes of slices of $G$.

\subsubsection{The shifted Functor }\label{shifted}
The last example of a Green biset functor that we will see is the shifted Functor by a  finite group $K$ (see  \cite{serge-biset} and \cite{simpleRomero}).\\
Let $K$ be a finite group. The   Green biset functor $A$ over $R\mathcal{D}$ can be shifted
by $K$. This gives a new Green biset functor, $A_K$,  defined for a finite group $G$
by
\begin{align*}
A_K(G) = A(G \times K) .
\end{align*}
For finite groups $G$ and $H$ and a finite $(H, G)$-biset $U$, the map
\begin{align*}
A_K(U) : A_K(G) \longrightarrow A_K(H)
\end{align*}
is the map $A(U \times K)$, where $U \times K$ is viewed as a $(H \times K, G \times K)$-biset in
the obvious way, letting $K$ act on both sides on $U \times  K$ by multiplication on
the second component. For an arbitrary element $\alpha \in  R B(H, G)$, that is a
$R$-linear combination of $(H, G)$-bisets, the map $A_K(\alpha) : A_K(G) \longrightarrow A_K(H)$ is
defined by $R$-linearity.
This endows $A_K$ with a biset functor structure. Moreover, for finite
groups $G$ and $H$, the product
\begin{align*}
\times_{A_K}: A_K(G) \times  A_K(H) \longrightarrow A_K(G \times  H)
\end{align*}
is defined as follows: if $\alpha \in  A_K(G) = A(G\times K)$ and $\beta \in A_K(H) = A(H\times K)$,
then $\alpha \times \beta \in  A(G \times K \times H \times K)$. We set
\begin{align*}
\alpha \times_{A_K} \beta =A(Iso(\delta) \circ Res^{G\times K\times H\times K}_\Delta) (\alpha \times \beta) ,
\end{align*}
 where $\Delta = \lbrace (g, k, h, k) | g \in G, h \in H, k \in K \rbrace$, and $\delta$ is the isomorphism
$\Delta \longrightarrow G \times H \times K $ sending $(g, k, h, k)$ to $(g, h, k)$. The identity element $\xi_{A_K}$ is $A(Inf^K_1)(\xi_A)$. This bilinear  product   makes $A_K$ a Green biset functor.

\section{Biset Operations on Idempotents } \label{biset and ideals}
\textbf{Hypothesis}:  In this section $ \mathbb{K}$ is a field of characteristic 0  and   $A$ is a Green biset functor such that $A(G)$ a finite dimensional split semisimple commutative  $\mathbb{K}$-algebra with the product of  definition \ref{green}, for all finite groups $G$. We denote by $E_G$ the corresponding unique set of orthogonal primitive idempotents of $ A (G)$, which exists by the above hypotheses.

\begin{lem}\label{indepontProducPropo}
Let $f \in  A(G)$ be  such that for all $v\in A(G)$, we have  $v\cdot f = \lambda_v f$  for some  $\lambda_v \in \mathbb{K}$, then  $f=ke$ for some $e\in E_G$ and  $k\in \mathbb{K}$.
\begin{proof}
By hypothesis over the functor $A$,  we have  that   $f$  is a linear  combination of elements of  $E_G$,  
\begin{align*}
f=\sum_{e\in E_G} k_e e, 
\end{align*}
where $k_e\in \mathbb{K}$.  If $f\neq 0$, then there exists  $e\in E_G$  such that  $k_e\neq 0$.  By orthogonality, 
\begin{align*}
e\cdot  f = \lambda_e f =k_e e\neq 0,  
\end{align*}
it follows that $\lambda_e\neq 0$. Suppose  there exists  $s\in E_G- \lbrace e \rbrace$  such that $k_s\neq 0$. Then  $s\cdot f =\lambda_s f$, with $\lambda_s \neq 0$, hence 
\begin{align*}
0=(s\cdot e)\cdot f =s\cdot(e \cdot f)=s\cdot (\lambda_ef)=\lambda_s\lambda_e f,
\end{align*}
thus $\lambda_s\lambda_e=0$, this contradiction shows that
 $f=\lambda_e e$.\\
If $f=0$, then $f=0 \cdot e$,  for any  $e \in E_G$.
\end{proof}
\end{lem}
Since our interest is to study the ideals generated by the orthogonal idempotents and    $A$ is a Green biset functor, by Lemma 2.3.26 of \cite{serge-biset}, it is   natural to look at the effect of elementary operations ($Ind$, $Res$, $Inf$, $Def$ and $Iso$) on the idempotents of the algebra $A(G)$.
\begin{thm} 
\label{teodeoper}
Let $G$ be a finite group. 
\begin{enumerate}[(a)]
\item Let $H$ be a subgroup of $G$, and $e\in E_G$. Then
\begin{align*}
A(Res^G_H)(e) =\sum_{s\in \Omega_H} s,
\end{align*} 
where  $\Omega_H$ is a subset of  $E_H$.
\item Let $H$  be a subgroup of $G$ and $t\in E_H$. Then there exist $e\in E_G$ and $\lambda\in \mathbb{K}$ such that
\begin{align*}
A(Ind^G_H)(t)=\lambda e.
\end{align*}
\item Let $N$ be a normal subgroup of $G$ and $e \in E_{G/N}.$ Then 
\begin{align*}
A(Inf_{G/N}^G)(e) = \sum_{t \in \Omega_{G/N}} t,
\end{align*}
where $\Omega_{G/N}$ is a subset of  $E_G$.
\item Let $N$ be a normal subgroup of $G$ and $e \in E_{G}$. Then there exist $\lambda \in \mathbb{K}$ and $t \in E_{G/N}$ such that 
\begin{align*}
A(Def_{G/N}^G) (e)= \lambda t.
\end{align*}
\item If $\phi : G \longrightarrow G^\prime$ is a group isomorphism  and $e\in E_G$. Then
\begin{align*}
A(Iso \phi)(e) = e^\prime
\end{align*}
for some $e^\prime \in E_{G^\prime}$.
\end{enumerate}
\end{thm}
\begin{proof}
\begin{enumerate}[(a)] 
\item Observe that $A(Res^G_K) : A(G)\longrightarrow  A(H)$ is a ring homomorphism. It follows that $A(Res^G_H )(e)$ is an idempotent of $A(H)$, hence a sum of elements of $E_H$.
\item  We have  that $A(Ind^G_H)(t) $  is a linear combination  of elements of $E_G$,
\begin{equation} \label{combilineal}
A(Ind_H^G) (t)=\sum_{e\in E_H} \lambda_e e.
\end{equation}
If $ A(Ind^G_H) (t) \neq 0$,    then there exist  $ e \in  E_G$ and  $\lambda_{e}\in  \mathbb{K}$ such that  $  \lambda_{e} \neq 0$, hence
\begin{align*}
 0 \neq e \cdot \left( A(Ind^G_H)(t)\right)  = \lambda_{e} e. 
\end{align*}
By  The Frobenius relations (Definition \ref{green}), this is equal to 
\begin{align*}
A(Ind^G_H)\left( (A(Res^G_H)e)\cdot t\right), 
\end{align*}
thus  $\left(  A(Res^G_H)(e)\right) \cdot t \neq 0$. Moreover,  $A(Res^G_H$) is a ring homomorphism,
then we have  that
\begin{align*}
 A(Res^G_H) (e)=\sum_{s \in \Omega_H } s,
\end{align*}
where $\Omega_H \subseteq E_H$,
this means that  $\left(  A(Res^G_H)(e)\right) \cdot t =t $. 
Now suppose that there exists  $e^\prime \in E_G$ with $e^\prime \neq e$,  such that  $\lambda_{e^\prime}\neq 0$ in  (\ref{combilineal}), so   $$\left( A(Res^G_H)(e^\prime)\right) \cdot  t = t$$ by the same reason.  Then  
\begin{align*}
A(Res^G_H)(e)  \cdot A(Res^G_H)(e^\prime) \cdot t= t.
\end{align*}
Since $A(Res^G_H)$ is a ring homomorphism, it follows that $$A(Res^G_H)(e \cdot e^\prime) \neq 0,$$
in particular  $e \cdot e^\prime \neq 0$, but this is a contradiction, then  $e = e^\prime$.
Thus  $A(Ind^G_H)(t)  = \lambda_e e$. \\
 On the other hand, if $A(Ind^G_H)(t) = 0$, we have  $A(Ind^G_H)(t)  = 0 \cdot e$  for all $e \in E_G$.
\item Since $A(Inf_{G/N}^G)$   is a ring homomorphism, one has that $A(Inf_{G/N}^G)(e)$ is an idempotent of $A(G)$, hence it  is a sum of elements of $E_G$. 
\item Let  $v\in A(G/N)$,  then   by the  Frobenius relations (Definition \ref{green}), it follows that
\begin{align*}
v\cdot A(Def^G_{G/N})(e)=A(Def_{G/N}^G)(A(Inf_{G/N}^G)(v)\cdot e)
\end{align*}
since  $A(Inf_{G/N}^G) (v)$ is an element of  $A(G)$, we have that $A(Inf_{G/N}^G)( v)$ is a linear combination of the elements of $E_G$, it follows that 
\begin{align*}
A(Inf_{G/N}^G)(v)\cdot e =\alpha \cdot e
\end{align*}
for some  $\alpha \in \mathbb{K}$. Thus, 
\begin{align*}
 v\cdot A(Def^G_{G/N})(e)=A(Def_{G/N}^G)(A(Inf_{G/N}^G)(v)\cdot e) =\alpha A(Def_{G/N}^G )(e).
\end{align*}
 By  Proposition \ref{indepontProducPropo}, there exists $t\in E_{G/N}$ and  $\lambda \in \mathbb{K}$ such that    $A(Def^G_{G/N})(e)=\lambda t$.
\item Observe that $A(Iso\phi)$  is a ring isomorphism.
\end{enumerate}
\end{proof}

\begin{defi}\label{A-group}
 Let $H$ be a finite group. We define
 \begin{align*}
 \underline{E_H }&= \lbrace e_H \in  E_H \mid A(Res^H_K)e_H = 0 \ \forall K < H\rbrace \text{ and} \\
\underline{\underline{E_H}}&=\lbrace e_H\in E_H \mid A(Def^H_{H/N})e_H=0 \ \forall \ N \unlhd H, \ N\neq 1\rbrace.
\end{align*}
We will say that   $H$ is an  $MC$-group  of $A$ (or just $MC$-group if there is no risk of confusion) if  $\underline{E_H}\cap \underline{\underline{E_H}}  \neq \emptyset $.
\end{defi}

\begin{lem}\label{Res-Ind}
 Let  $G$ be a finite group and  $e_G \in E_G$.  Let $H$ be  a minimal subgroup of $G$  with respect to 
$A(Res^G_H) (e_G )\neq 0$ (this group exists because $A(Res^G_G)(e_G) \neq 0$). Then there exist $e_H\in  \underline{E_H}$ and $\alpha \in \mathbb{K}$ such that $e_G = \alpha A(Ind^G_H)(e_H)$.
\end{lem}
\begin{proof}
Since  $A(Res^G_H)(e_G )\neq 0$, by (a) of Theorem \ref{teodeoper} there exists  $e_H \in E_H$ such that $$e_H \cdot A(Res^G_H)(e_G) =e_H.$$ Moreover, if  $K$ is a subgroup of  $H$, we have
\begin{align*}
A(Res^H_K)(e_H)\cdot A(Res^G_K)(e_G) = A(Res^H_K)e_H.
\end{align*}
If $K\neq H$, then  $A(Res^G_K)(e_G) = 0$ and it follows that $A(Res^H_K)(e_H )= 0$. This means that $e_H \in \underline{E_H}$.\\
Now suppose that $A(Ind^G_H)(e_H )= 0$. By the Mackey formula (Relation 1.1.3 of \cite{serge-biset}) we have 
\begin{align*}
0 = A(Res^G_H \circ  Ind^G_H)(e_H) = \sum_{x\in  [H\setminus G/H]} A(Ind^H_{H\cap ^xH} \circ  Iso(\gamma_x) \circ  Res^H_ {H^x\cap H})(e_H),
\end{align*}
where $\gamma_x $  is the conjugation isomorphism by  $x$   from $H^x \cap H$ to  $H \cap ^xH.$  If 
$H^x \cap H$ is a  proper subgroup of  $H$, then $A(Res^H_{H^x\cap H})e_H = 0$ and it follows that the last equation is equal to 
\begin{align*}
0 = \sum_{ x\in [N_G(H)/H]} A(Iso(\gamma_x))(e_H).
\end{align*}
Since $e_H \in E_H$, we have that  $A(Iso(\gamma_x))(e_H) \in E_H$  for all $x \in N_G(H)$, because 
$A(Iso(\gamma_x))$ is a  ring isomorphism  on $A(H)$. Now multiplying by $e_H$, 

\begin{align*}
0 &= e_H \sum_{x\in [N_G(H)/H]} A(Iso(\gamma_x))(e_H)\\
 & = |\lbrace x \in [N_G(H)/H] \mid  A(Iso(\gamma_x))(e_H) = e_H \rbrace |\cdot e_H.
\end{align*}
Since the set on the right is not empty, one has $e_H =0$. This contradiction shows that 
$A(Ind^G_H)e_H \neq 0$.  Since  $e_H \cdot A(Res^G_H)(e_G) =e_H$, then by the Frobenius identities (Definition \ref{green})
\begin{align*}
A(Ind^G_H)(e_H) \cdot e_G =A(Ind^G_H)( e_H\cdot   A(Res^G_H)(e_G)  ) =  A(Ind^G_H)(e_H)\neq 0. 
\end{align*}
Now,  by the proof Theorem \ref{teodeoper},   there exists $\lambda \neq 0$,  such that $$A(Ind^G_H)(e_H) = \lambda e_G.$$  If $\alpha = 1/\lambda$ one has $e_G = \alpha A(Ind^G_H)(e_H)$.
\end{proof}

\begin{lem} \label{lemadefla}
Let  $G$ be a finite group and $e_G\in E_{G}$. If $N$ is a normal subgroup of $G$, maximal   with respect to   $A(Def^G_{G/N})(e_G)\neq 0$, then there exists  $e_{G/N}\in \underline{\underline{E_{G/N}}} $ and  $\alpha \in \mathbb{K}$ with $\alpha \neq 0$,  such that  $e_{G/N}=\alpha  A(Def_{G/N}^G)(e_G) $ and $e_G= e_G \cdot A(Inf_{G/N}^G )(e_{G/N})$.
\begin{proof}
Let  $e_G\in E_G$ and   $N$ a normal subgroup of $G$, maximal with the propriety  $A(Def^G_{G/N})e_G\neq 0$ (this  group  exists because $A(Def^G_{G/1})e_G\neq 0$). By Theorem \ref{teodeoper} (d), we have 
\begin{align*}
A(Def^G_{G/N})(e_G)= \lambda e_{G/N},
\end{align*}
where  $e_{G/N}\in E_{G/N}$ and  $0 \neq\lambda \in \mathbb{K}$, then $e_{G/N}=(1/\lambda) A(Def_{G/N}^G)(e_G) $.  On the other hand, we know by Theorem \ref{teodeoper} (c) that
\begin{align*}
A(Inf_{G/N}^G)(e_{G/N})=\sum_{t \in \Omega} t 
\end{align*}
where $\Omega$ is a subset of  $E_G$. Now suppose that $e_G \cdot  A(Inf_{G/N}^G)(e_{G/N})=0$,
then by the Frobenius identities (Definition \ref{green}) 
\begin{align*}
0=A(Def_{G/N}^G) \left( e_G\cdot  A(Inf_{G/N}^G)(e_{G/N})\right) &=A(Def_{G/N}^G)( e_G)  \cdot  e_{G/N}\\
&= (\lambda e_{G/N})\cdot  e_{G/N}\\
&= \lambda e_{G/N},
\end{align*}
this contradiction shows that  $e_G \cdot  A(Inf_{G/N}^G)(e_{G/N} )\neq 0$. This means that $e_G$ is an element of $\Omega$, thus $e_G \cdot A (Inf^G_{G/N})( e_{G/N} )= e_G$. \\
Let $M/N$  be a non-trivial normal subgroup of $G/N$. Then
 \begin{align*}
A(Def^{G/N}_{(G/N)/(M/N)}) (e_{G/N} )&= A(Def^{G/N}_{(G/N)/(M/N)} \circ Def^G_{G/N})((1/\lambda) e_G) \\
&= 1/\lambda \cdot A(Def^{G}_{(G/N)/(M/N)} )( e_G)\\
&=  1/\lambda \cdot A(Def^{G}_{G/M} ) (e_G)\\
&=0,
 \end{align*}
 thus  $e_{G/N} \in \underline{\underline{E_{G/N}}}. $
\end{proof}
\end{lem}

\begin{nota}
If $G$ is a finite group, denote by  $\textbf{e}_\textbf{G}$  the ideal 
of $A$ generated by $e_G\in E_G$.
\end{nota}
\begin{lem}\label{lema menor}
Let $H$ and $K$ be  finite groups, and let  $e_H\in E_H$ and  $e_K\in E_K$. We have $\emph{\textbf{e}}_\emph{\textbf{H}}\subseteq \emph{\textbf{e}}_\emph{\textbf{K}}$  if only if  there exists \texttt{$e_{H\times K} \in E_{H\times K}$} such that   
\begin{align*}
e_H \cdot A( Def^{H\times K}_{ H}) \left( e_{H\times K} \cdot A( Inf^{H\times K }_{K} ) (e_K)\right) \neq 0 
\end{align*}

\end{lem}
\begin{proof}
$\Rightarrow ]$
By Remark \ref{inter},  we have that 
\begin{align*}
e_H\in \textbf{e}_\textbf{K}(H)= Hom_{\mathcal{P}_A}(K,H)e_K
\end{align*}
where $Hom_{\mathcal{P}_A}(K,H)=A(H\times K)$. Now, by the  hypothesis over  $A$, there exists  $e_{H\times K} \in E_{H\times K}$  such that 
\begin{equation} \label{e_K(H)}
e_H\cdot(e_{H\times K }\circ e_k)\neq 0.
\end{equation}
By Lemma \ref{a comp b} with $G=1$, one has 
\begin{align*}
e_H \cdot A(Def^{H\times K}_H)\left( e_{H\times K} \cdot A (Inf^{H\times K}_K)(e_K) \right) \neq 0.
\end{align*}
$\Leftarrow ]$ By Proposition 8.6.1 of \cite{serge-biset} and  Remark 3.2.9 of \cite{serge-biset}, we have $e_H \in \textbf{e}_\textbf{K} (H)$.
\end{proof}

\begin{lem} \label{res-ideal}
Let  $G$ be a finite group and   $e_G \in E_G$, then there exist a subgroup $H$ of  $G$ and $e_H \in \underline{E_H}$, such that $\emph{\textbf{e}}_\emph{\textbf{G}} =\emph{\textbf{e}}_\emph{\textbf{H}}$.
\end{lem}
\begin{proof}
Let  $e_G \in E_G$.  By Lemma \ref{Res-Ind}, there  exist  a subgroup  $H$ of  $G$, $e_H \in \underline{E_H}$ and $\alpha \in \mathbb{K}$, such that 
\begin{align*}
e_H\cdot A(Res^G_H)(e_G)&= e_H\\
\alpha \cdot A(Ind^G_H)(e_H)&=e_G.
\end{align*} 
This means that, 
\begin{align*}
e_H=e_H\cdot A(Res^G_H)(e_G) \in \textbf{e}_\textbf{G}(H)
\end{align*}
then  $\textbf{e}_\textbf{H}$ is an $A$-submodule of  $\textbf{e}_\textbf{G}$. On the other hand,
\begin{align*}
e_G=\alpha \cdot A(Ind^G_H)(e_H) \in \textbf{e}_\textbf{H}(G),
\end{align*}
then  $\textbf{e}_\textbf{G}$ is an $A$-submodule of  $\textbf{e}_\textbf{H}$.  Thus  $\textbf{e}_\textbf{G}=\textbf{e}_\textbf{H}$.
\end{proof}

\begin{lem} \label{defl-ideal}
Let  $H$ be  a finite group  and  $e_H\in E_H$, then there exist  a normal subgroup $N$ of $H$  and  $e_{H/N}\in \underline{\underline{E_{H/N}}}$, such that  $\emph{\textbf{e}}_\emph{\textbf{H} }= \emph{\textbf{e}}_\emph{\textbf{H/N}}$.
\end{lem} 
\begin{proof}
By Lemma \ref{lemadefla},  there exist    a normal subgroup $N$ of  $H$,  $e_{H/N} \in \underline{\underline{E_{H/N}}}$ and $0 \neq\lambda \in \mathbb{ K}$, such that  
\begin{align*}
e_H &=e_H \cdot A(Inf_{H/N}^H)(e_{H/N}),\\
e_{H/N} &=  \lambda A(Def_{H/N}^H)(e_H).
\end{align*}
 This means that   $e_{H/N} \in \textbf{e}_\textbf{H} (H/N)$, it follows that $ \textbf{e}_\textbf{H/N} \subseteq \textbf{e}_\textbf{H}$.\\
  On the other hand  $e_H \in \textbf{e}_\textbf{H/N} (H)$, then $\textbf{e}_\textbf{H} \subseteq \textbf{e}_\textbf{H/N}$. Thus  $ \textbf{e}_\textbf{H/N} = \textbf{e}_\textbf{H}$.
\end{proof}

\begin{lem} \label{Ideal-Agrupo}
Let  $G$  be a finite group  and $e_G\in E_G$, then there exist an $MC$-group  $H$ and  $e_H$  element of   $ \underline{E_H} \cap \underline{\underline{E_H}}$, such that  $\emph{\textbf{e}}_\emph{\textbf{G}}=\emph{\textbf{e}}_\emph{\textbf{H}}$.
\end{lem}
\begin{proof}
By Lemma \ref{res-ideal},  there exist   $K\leq G$ and $e_{K} \in \underline{E_{K}}$  such that   $\textbf{e}_\textbf{G}= \textbf{e}_\textbf{K}$. On the other hand, by Lemma \ref{defl-ideal}, there exist $N \trianglelefteq K$, $e_{ K/N}\in \underline{\underline{E_{K/N}}}$ and $0\neq \lambda \in \mathbb{K}$, such that  $e_{K/N}=\lambda  A(Def^K_{K/N})(e_{K}) $,  and  $\textbf{e}_\textbf{K/N}= \textbf{e}_\textbf{K}= \textbf{e}_\textbf{G}$. 

Let  $M/N  \lneq K/N$, by the relation 1.1.3 (2-f) of  \cite{serge-biset}, we have
\begin{align*}
A(Res_{M/N} ^{K/N}) e_{ K/N} &= A(Res_{M/N} ^{K/N}) \lambda A(Def ^K_{K/N})(e_K)\\
&= \lambda \cdot A(Res_{M/N} ^{K/N} \circ Def ^K_{K/N})(e_K)\\
&=\lambda \cdot A(Def^M_{M/N} \circ Res_M^K) )(e_K)\\
&=0
\end{align*}
since $  A(Res_M^K) (e_K) = 0 $. Thus $e_{K/N} \in \underline{E_{K/N}} \cap \underline{\underline{E_{K/N}}}$.

\end{proof}

\begin{lem}
Let $F$ be an ideal of $A$, and $H$  minimal group for $F$ i.e. $F(H)\neq 0$ and $F(T)=0$ for any finite group $T$  such that $|T|<|H|$, then  $H$ is an  $MC$-group. 
\end{lem}
\begin{proof}
By hypothesis over $A$, we know
\begin{align*}
F(H)=\sum_{e_H \in \Omega_{F,H} } \mathbb{K} e_H \neq 0 
\end{align*}
where $\Omega_{F,H}$ is a subset of $E_H$, then there exists $e_H \in E_H$  such that  $e_H$ is an element of  $F(H)$. If  $N$ is a non-trivial normal subgroup of   $H$, we have
\begin{align*}
A(Def^H_{H/N})(e_H) \in F(H/N).
\end{align*}
Since $H$ is a minimal group for  $F$, it follows that  $F(H/N)=0$, thus  $$A(Def^H_{H/N})(e_H)=0,$$  this means  that $e_H \in \underline{\underline{E_H}}$. On the other hand, for any proper subgroup $K$ of $H$, one has  $A(Res^H_K) (e_H) \in F(K)=  0 $, then  $ A(Res^H_K) (e_H) =0$. It follows that  $e_H \in \underline{E_H}$. Thus, the idempotent $e_H $ is an element of  $ \underline{E_H}\cap \underline{\underline{E_H}}$, this is equivalent to say that $H$ is an  $MC$-group.
\end{proof}

\section{The ideals of the Green biset functor $A$}  \label{ideals and idem}
Recall:  $ \mathbb{K}$ is a field with characteristic 0  and   $A$ is a Green biset functor such that $A(G)$ is a finite dimensional split semisimple commutative    $\mathbb{K}$-algebra, for any finite group $G$. In  this section, we will show that any ideal $I$ of $A$ is the sum of some ideals generated by idempotents.
\begin{nota}
Let $I$ be an ideal of $A$, we will define $\mathcal{A}_{I}$ by:
\begin{align*}
\mathcal{A}_I:=\lbrace (H, e_H) \mid  H \text{ is an } MC\text{-group of } A, \ e_H \in  \underline{E_H} \cap \underline{\underline{E_H}} \text{ and } e_H \in I(H) \rbrace.
\end{align*}
\end{nota}

\begin{lem} \label{ideal=suma}
Let $I$ be an ideal of $A$, then 
$$
I=\sum_{ (H, e_H) \in \mathcal{A}_I} \emph{\textbf{e}}_\emph{\textbf{H}}.
$$
\end{lem}
\begin{proof}

By definition of $\mathcal{A}_I$, one has $\textbf{e}_\textbf{H} \leq I$ for all $(e_H, H) \in \mathcal{A}_I$. As a consequence, one obtains   $$\sum_{e_H \in \mathcal{A}_I}\textbf{e}_\textbf{H} \leq I.$$ On the other hand, by the hypothesis over $A$, we have  for a finite group $G$, 
$$I(G)= \sum_{ e_G\in \Omega_{I,G}} \mathbb{K} e_G,$$
where $\Omega_{I,G} $ is a subset of $E_G$.  Then $\textbf{ e}_\textbf{G }\leq I$ for all $e_G \in \Omega_{I,G}$. Moreover,  by  Lemma \ref{Ideal-Agrupo}, for all $e_G \in \Omega_{I, G}$  there exists an $MC$-group $K$ and  $e_K \in \underline{E_{K}}\cap \underline{\underline{E_{K}}}$,  such that  $\textbf{e}_\textbf{G} =\textbf{e}_\textbf{K}$. It follows that $\textbf{e}_\textbf{K}  \leq I$, this means that $e_K\in I(K)$,  therefore $(K, e_K) \in \mathcal{A}_I$ and 
 \begin{align*}
 I(G)&= \sum_{ e_G\in \Omega_{I,G}} \mathbb{K} e_G \subseteq \sum_{ e_G\in \Omega_{I,G}} \textbf{e}_\textbf{G}(G)\\
 &= \sum_{ e_G\in \Omega_{I,G}} \textbf{e}_\textbf{K}(G) \leq \sum_{(K, e_K) \in \mathcal{A}_I} \textbf{e}_\textbf{K}(G). 
 \end{align*}

Thus $I \leq \sum_{(K, e_K) \in \mathcal{A}_I} \textbf{e}_\textbf{K}$.

\end{proof}
 
 We define  the class $\mathcal{M}$ by 
 \begin{align*}
 \mathcal{M}:=\lbrace (H, e_H) \mid  H \text{ is an } MC\text{-group  of $A$ and }  e_H\in \underline{E_H} \cap \underline{\underline{E_H}}\rbrace
 \end{align*}

\begin{defi} \label{gg}
 We define the  relation  $\gg$ on $\mathcal{M}$, by $(H, e_H) \gg (K, e_K) $ if only if there exists  $e_{H\times K} \in E_{H\times K}$ such that 
\begin{align*}
e_H \cdot A( Def^{H\times K}_{ H}) \left( e_{H\times K} \cdot A( Inf^{H\times K }_{K} ) (e_K)\right) \neq 0.
\end{align*}
By Lemma \ref{lema menor}, one has  that  $(H, e_H) \gg (K, e_K) $  is equivalent to $ \emph{\textbf{e}}_{\textbf{H}}\subseteq \emph{\textbf{e}}_{\textbf{K}} $, then $(\mathcal{M}, \gg)$ is a preorder. By Chapter 3, Section  1.2. of \cite{bourbaki},  one has an equivalence relation on $\mathcal{M}$, definite by,  for $(H, e_H),\  ( K, e_K)\in \mathcal{M}$.  We say that   $ (H, e_H)$ and $(K, e_K)$ are equivalent if, $( H, e_H) \gg (K, e_K)$ and $ (K, e_K) \gg (H, e_H)$. In this case, we also write \- $(H, e_H)\sim (K, e_K)$. 
\end{defi}
Let $[H, e_H]$ denote  the  equivalence class of  $(H, e_H)\in \mathcal{M}$. In   Remark  \ref{setM} we prove that the equivalence classes of $\mathcal{M}$ form a set and we will denote by    $[\mathcal{M}]$.  Then $ ([\mathcal{M}], \gg)$  is a poset (see \cite{bourbaki}).
\begin{rem}\label{setM}
Let $S$ be a set of representatives of the isomorphism classes of finite groups.  We consider the set  $\mathcal{A}=  \prod_{G \in S} \mathcal{P}(A(G)$  where $\mathcal{P}(A(G))$ denotes the power set of $A(G)$ .  We define the function  
\begin{align*}
\varphi: \mathcal{M} &\longrightarrow \mathcal{A}\\
(H, 	e_H)&\longmapsto (\emph{\textbf{e}}_\textbf{H}(G))_{G\in S}
\end{align*}
Let $(H, e_H) , \  (K,e_K) \in  \mathcal{M} $.  One has   $(H, e_H) \sim (K,e_K)$ if only if $\emph{\textbf{e}}_\textbf{H} = \emph{\textbf{e}}_\textbf{K}$, this means   $\varphi((H, e_H))=\varphi((K,e_K))$. Thus  the function $\varphi :[\mathcal{M}]\longrightarrow \mathcal{A}$ is well defined.  Since, if  two groups $G$ and $H$  are 
isomorphic  as groups then they are isomorphic in $\mathcal{P}_A$  (see Notation 4.1 of  \cite{simpleRomero})  is easy to see that if two ideals coincide on all elements of $S$,  then  they are equal, thus $\varphi$ is injective  i.e. $[\mathcal{M}]$ is a set.
\end{rem}

A subclass 
 $\mathcal{N}$ of $\mathcal{M}$ is said to be closed if  
\begin{align*}
\forall  (H, e_H), \ (K, e_K) \in \mathcal{M},  \ (H, e_H) \gg (K, e_K) \in \mathcal{N} \Rightarrow (H, e_H)\in \mathcal{N}.
\end{align*}
A subset $B$ of $[\mathcal{M}]$ is a closed subset if there exists a closed subclass $\mathcal{N}$ of  $\mathcal{M}$ such that  $B=\mathcal{N} \cap [\mathcal{M}]$.

\begin{thm} \label{thm clasific}
Let  $\mathcal{S}_A$ denote the set of  ideals  of  $A$ (in the proof of this theorem we can see why  $\mathcal{S}_A$ is a set) , ordered by the inclusion  of subfunctors   and  let   $Cl_A$ denote  the set of closed subsets of $[\mathcal{M}]$, ordered by the  inclusion of subsets.\\
Then the map 
\begin{align*}
\Theta: I\longmapsto \lbrace [H, e_H]\in [\mathcal{M}] \mid e_H \in I(H) \rbrace
\end{align*}
is an isomorphism of posets from $\mathcal{S}_A$ to  $Cl_A$. The inverse isomorphism is
the map
\begin{align*}
\Psi: Cl_A &\longrightarrow \mathcal{S}_A\\
B &\longmapsto \sum_{[H, e_H]\in B} \emph{\textbf{e}}_\emph{\textbf{H}}.
\end{align*}
\end{thm}

\begin{proof}
Let $I$ be an ideal of $A$. First, show that  $\Theta(I) = [\mathcal{A}_I]:=[\mathcal{M}] \cap \mathcal{A}_I$:
\begin{itemize}
\item  Let $ [H, e_H]\in [\mathcal{M}] \cap  \mathcal{A}_I$. By definition of $\mathcal{A}_I$,  we have  $e_H \in I(H)$, this means that  $[H, e_H]\in \Theta (I)$,  thus $[\mathcal{M}] \cap \mathcal{A}_I\subseteq\Theta(I) $
\item On the other hand,  for all  $[H, e_H]\in \Theta (I)$, one has  $e_H \in I(H)$ and  $[H, e_H]\in [\mathcal{M}]$. It follows that, $H$ is an  $MC$-group and  $e_H \in \underline{E_H} \cap \underline{\underline{E_H}} $, then $[H, e_H] \in [\mathcal{M}] \cap \mathcal{A}_I$. It follows that  $\Theta(I) \subseteq[\mathcal{M}] \cap \mathcal{A}_I$.
\end{itemize}
We next show that $[\mathcal{A}_I]$ is a closed subset.  Let $[H, e_H], \ [K, e_K] \in [\mathcal{M}], $  such that   $[H, e_H] \gg [K, e_K] \in [\mathcal{A}_I] $. Then  $\textbf{e}_\textbf{H} \subseteq \textbf{e}_\textbf{K} \subseteq I$, thus $e_H\in I(H)$ .i.e. $[H, e_H]\in [\mathcal{A}_I]$. \\
 Next we show that the map $\Theta$ is a map of posets. Let  $F$ and  $I$ be ideals of $A$ such that $F\subseteq I$.  Let  $[H, e_H ]\in [\mathcal{A}_F] $,  this means that  $e_H \in F(H)$, as a consequence  $e_H \in I(H)$, therefore $[H, e_H ] \in [\mathcal{A}_I] $. On the other hand, the map $\Psi$
is also obviously a map of posets. \\
 Finally, we show that $\Psi \circ \Theta = 1_{\mathcal{S}_A}$ and $\Theta \circ \Psi = 1_{Cl_A}$.
By Lemma \ref{ideal=suma}, we have  
\begin{align*}
I=\sum_{(K, e_K) \in \mathcal{ A}_I} \textbf{e}_\textbf{K}.
\end{align*}
 and by  Lemma \ref{lema menor}, one has that $\textbf{e}_\textbf{K}=\textbf{e}_\textbf{H}$ if only if  $(H, e_H) \sim (K, e_K)$, it follows that
\begin{align*}
\sum_{(K, e_K) \in \mathcal{ A}_I} \textbf{e}_\textbf{K}=  \sum_{[K, e_K] \in [\mathcal{A}_I]} \textbf{e}_\textbf{K}.
\end{align*}
Therefore  $I=\Psi(\Theta (I))$.\\
On the other hand,  let $B \in Cl_A$. Then,
\begin{align*}
\Theta(\Psi(B)) = \lbrace [H, e_H]\in [\mathcal{M}] \mid e_H \in \sum_{[K, e_K]\in B} \textbf{e}_\textbf{K}(H) \rbrace.
\end{align*}
So, obviously $B \subseteq \Theta(\Psi(B)) $. Conversely, if  $[H, e_H] \in \Theta(\Psi(B))$,  then
\begin{align*}
e_H \in \sum_{[K, e_K]\in B} \textbf{e}_\textbf{K}(H).
\end{align*}
 So there exists  $[K, e_K]\in B$ such that $e_H \cdot \textbf{e}_\textbf{K}(H) \neq 0$, i.e.  $e_H\in \textbf{e}_\textbf{K}(H)$.  Then   $\textbf{e}_\textbf{H} \subseteq \textbf{e}_\textbf{K}$ and by  Lemma \ref{lema menor}, it follows that   $[H, e_H]\gg [K, e_K]$.  Since $B$ is a closet subset, one has  $[H, e_H]\in B$. Then $B=\Theta(\Psi(B))$.
\end{proof}
Note that the poset  $Cl_A$ is a lattice, where the  join is the  union of sets and the meet is the  intersection of sets. \\
Let $F$, $J  \in \mathcal{S}_A$. We define  the sum of ideals $F$ and $J$, denoted by $F+J$, is the functor such that  $F+J (G):=F(G)+J(G)$ in any finite group $G$, and for any finites groups $G$ and  $H$ ,  $\alpha \in Hom_\mathcal{C}(G,H)$, $a\in F(G)$ and $b\in J(G)$,  we define $F+J(\alpha) (a+b):= F(\alpha)(a)+J(\alpha)(b)$, is easy to see the functor $F+J$ is an element of $\mathcal{S}_A$. 
We define the product of the ideals $F$ and $J$, denoted by $FJ$,  like the functor such that 
\begin{cor}
The lattice  $\mathcal{S}_A$ is  a distributive lattice.
\end{cor}
\begin{proof}
The  lattice  $\mathcal{S}_A$  is distributive, since  the lattice $Cl_A$ is clearly distributive, where the join is the sum of ideals  we define   and the meet is the interception  of ideals. 
\end{proof}
\section{Examples } \label{examples}
In this section, we look at several examples of Green biset functors  such that every evaluation in a finite group   is a finite dimensional split semisimple commutative algebra, and  we will compare the relationship $ \gg $  with the equivalence relations given in each of these  examples by previous authors.

 \subsection{The Burnside Functor }
The first example  is the Burnside functor.  This functor was  studied by Serge Bouc  in \cite{serge-biset}.  By the ghost map (see subsection  2.5.1 of \cite{serge-biset}), one has that  $\mathbb{K} B(G) := \mathbb{K} \otimes_ \mathbb{Z} B(G)$ is  a finite dimensional split semisimple commutative $\mathbb{K}$-algebra, where  $ \mathbb{K}$ is a field with characteristic 0. Then the Burnside functor satisfies our hypothesis.

We recall some results that will be useful when comparing the classifications of ideals.
\begin{thm}[Theorem 2.5.2 of \cite{serge-biset}]
Let $G$ be a finite group. If $H$ is a subgroup of G, denote by $e^G_H$ the element of $\mathbb{K}B(G)$ defined by
\begin{align*}
e^G_H = \frac{1}{|N_G(H)|} \sum_{K\leq H} |K| \mu (K, H) [G/K],
\end{align*}
where $\mu$ is the Möbius function of the poset of subgroups of $G$.\\
Then $e^G_H = e^G_K$ if the subgroups $H$ and $K$ are conjugate in $G$, and the
elements $e^G_H$,  when $H$ runs through a set $ [S(G)]$ of representatives of conjugacy classes of
subgroups of $G$, are the orthogonal  primitive idempotents of the $\mathbb{K}$-algebra
$\mathbb{K}B(G)$.
\end{thm}
If $G$ is a finite group, and $N$ is a normal subgroup of $G$,
denote by $m_{G,N}$  the rational number defined by
\begin{align*}
m_{G,N} = \frac{1}{|G|} \sum_{XN=G}  |X| \mu(X, G),
\end{align*}
where $\mu$ is the Möbius function of the poset of subgroups of $G$. 

\begin{defi} [Definition 5.4.6. of \cite{serge-biset}]
A finite group $G$ is called a $B$-group  if $|G| \neq  0$,  and if for any non-trivial normal subgroup $N$ of $G$, the constant $m_{G,N}$
is equal to zero,  or equivalently $B(Def^G_{G/N} )(e^G_G) = 0$ in $\mathbb{K}B(G/N)$.
The class of all $B$-groups  is denoted by $B$-$grp$.
\end{defi}
\begin{prop}
Let $G$ be a finite group. One has that $G$ is a   $B$-group if only if $G$ is an   $MC$-group.
\end{prop}
\begin{proof}
If  $G$  is  a $B$-group,  we have  $B(Def^G_{G/N}) (e^G_G)=0$, for $1\neq N \unlhd G$ i.e. $e^G_G \in \underline{\underline{E_G}}$. Also,  by  Proposition  5.4.5. of \cite{serge-biset},  $G$ is a minimal group for ideal generated by $e^G_G$,  then  $e^G_G \in \underline{E_G}$. Thus   $e^G_G \in \underline{\underline{E_G}} \cap \underline{E_G}$, i.e. $G$ is an $MC$-group of $B$ (Definition \ref{A-group}).\\
On the other hand, let $G$ be an $MC$-group. This means that there exist  $K\leq G$ such that $e^G_K\in \underline{E_G} \cap \underline{\underline{E_G}}$. If $K$ is a proper subgroup of $G$,  by Theorem 5.2.4 of $\cite{serge-biset}$,  one has 
\begin{align*}
B(Res^G_K)( e^G_K)=e^K_K,  
\end{align*}
but $e^G_K \in \underline{E_G}$, then $B(Res^G_K) (e^G_K)=0$.  This
contradiction shows that $K=G$. i.e.  $e^G_G\in  \underline{E_G} \cap \underline{\underline{E_G}}$. It follows that $ B(Def^G_{G/N}) (e^G_G)=0$ for any  $1\neq N \unlhd G$. Thus $G$ is a $B$-group.
\end{proof}
\begin{defi}[Definition  5.4.13 of \cite{serge-biset}]
We can define a relation  $\gg_B$ on $B$-$grp$, by $G \gg_B  H$ if and only if $H$ is isomorphic to a quotient of $G$.
\end{defi}
Now we show that the relation $\gg_B$  is equivalent to  the relation $\gg$.
\begin{lem}
Let  $\gg_B$ denote  the relation in  Definition 5.4.13 of \cite{serge-biset}. Let $K$ and $H$ be $B$-groups. Then,  $(K, e^K_K)\gg (H, e^H_H)$ if only if  $K\gg_B H$.
\end{lem}
\begin{proof}
By Proposition 5.4.8 in \cite{serge-biset}, $K \gg_B H$ if and only if $\textbf{e}_\textbf{K} \subseteq\textbf{ e}_\textbf{H}.$ But, by
Lemma \ref{lema menor} this is equivalent to $(K, e_K^K) \gg (H, e_H^H).$
\end{proof}
 Since the relations  $\gg_B$ and $\gg$ are equivalent,  then  Theorem 5.4.14 of \cite{serge-biset} is equivalent to the Theorem \ref{thm clasific}.

 \subsection{Fibered $p$-biset Functor}
The second example  we will see is the fibered Burnside functor, this functor was  studied in \cite{fibered} by  Robert Boltje and Olcay Co{\c{s}}kun.  We defined this  functor in subsection \ref{sect fibered functor}. The precise situation is as follows.  Let $\mathbb{K}$ a be a sufficiently large field of characteristic $q$, and  we fix a prime $p\neq q$ and a positive integer $n$. For an abelian group $A$,  we denote by  and  $O(G) = O_A(G)$ the intersection of kernels of all homomorphism $G\longrightarrow A$ and $\xi_G(A)$  is the set of all pairs $(H, h)$ where $H \leq G$ and $h$ runs
over a complete set of left coset representatives of $O_A(H $) in $H$. The set $\xi_G(A)$ is a $G$-set
via conjugation and we denote by $[H,h]_G$ the  conjugacy class of $(H,h)$.  Let $\mathcal{M}_G(A)=\lbrace (V,\lambda ) \mid V\leq G, \ \lambda \in Hom_{grp} (V, A) \rbrace$, then by Lemma 3.1 of \cite{barker}  one  has  $|\mathcal{M}_G(A)|=|\xi_G(A)|$ and  by Theorem 5.2 of \cite{barker} there exists   a bijective correspondence between a set of conjugacy classes of $\xi_G(A)$ and the set of primitive idempotents of $\mathbb{K}B^A (G)$,  where the class   $[H,h]_G$  corresponds to the idempotent
\begin{align*}
e^G_{H,h}= \frac{1}{|N_G(H,h)|} \sum_{(V, \nu) \in \mathcal{M}_A(G)/G} |V| \mu_G \left( V, \nu; H,h\right) [V,\nu]_G,
\end{align*}
where $N_G (H, h)$ denotes the stabilizer in $G$ of the pair $(H, h)$ under the above action, 
\begin{align*}
\mu_G \left( V, \nu; H,h\right) =\sum_{(V^\prime,\nu^\prime) \in [V,\nu]_G} \nu^{-1}\left( V\cap hO_A(H)\right) \mu(V^\prime, H)/|V^\prime|
\end{align*}
is the monomial Möbius function and the above sum is over all pairs $(V^\prime,\nu^\prime)$ $G$-conjugate
to $(V , \nu)$.

\begin{prop}[Proposition 7 of \cite{cocskun}] \label{res-ind-fiber}
Let $N \unlhd G$ be finite groups.  Then  for $(H, h) \in \xi_G(A)$, we have
\begin{align*}
Def^G_{G/N} e^G_{H,h} =m_{H,h}^G \cdot e^{G/N}_{HN/N,hN}
\end{align*}
for some constant $m^G_{H,h} \in \mathbb{K}$.

\end{prop}
 We denote by $\mu_n$ a cyclic group of order $p^n$.  For a $p$-group $P$, denote  by $ O_n(P )$ the group $O_{\mu_n} (P )$ and the functor  $\mathbb{K}B^{\mu_n}$ by $\mathbb{K}B^n$. In this case, the group $ G^\ast$ is isomorphic to the dual group $Hom(G/O(G), \mu_n)$.
Let $\Phi(G)$ denote the Frattini subgroup of $G$.

\begin{prop}[Lemma 2 of \cite{cocskun}]\label{def-frattini-fiber}
For any $p$-group $G$ and $g\in  G$, we have
\begin{align*}
Def^G_{G/\Phi(G)} e^G_{G,g}=\frac{|O_n(G)|}{|N_G(G,g)|} \cdot |G/\Phi(G)|\cdot e^{G/\Phi(G)}_{G/\Phi(G),g\Phi(G)}.
\end{align*}
\label{m H g}
Let $G$ be an elementary abelian $p$-group of order $p^r$,  and $h$ be a non-trivial
element of $G$. Let  $H =\langle h \rangle$ be the subgroup generated by $h$,  we have
\[m_{H,g} = \left\{
  \begin{array}{lr}
    \frac{1-p^{r-1}}{p} &  \text{if }  g=1,\\
    \frac{1}{p} &  g \neq 1, \ g\in H,\\
     \frac{1-p^{r-2}}{p} &  g\notin H .
  \end{array}
\right.
\]
\end{prop}
 Let $\mathcal{I}=\lbrace 0 \rbrace \cup \lbrace r\in \mathbb{N} \mid p^{r-1}\equiv 1 (mod q ) \rbrace$.
\begin{thm} [Theorem 12 of \cite{cocskun}]
 Let $F$ be an ideal of $\mathbb{K}B^n$ and $G$ be a minimal group of $F$. Then,
 \begin{enumerate}[(i)]
 \item the group $G$ is elementary abelian of order $p^r$, for some $r \in\mathcal{I}$,
 \item the $\mathbb{K}$-vector space $F (G)$ is $1$-dimensional generated by $e^G_{G,1}$,
 \item  the ideal F is generated by $e^G_{G,1}$.
 \end{enumerate}

\end{thm}

\begin{prop}
Let  $G$ be a $p$-group of order  $p^r$. The group $G$ is an  $MC$-group of $ \mathbb{K}B^n$ if only if  it is elementary abelian  and  $p^{r-1} = 1 $ (mod $ q$).
\end{prop}
\begin{proof}
First, suppose $G$ is an $MC$-group of $ \mathbb{K}B^n$, then  there exists  an idempotent $e_{H,h}^G\in\underline{E_G} \cap  \underline{\underline{E_G}}$. If $H\lneq G$, by Theorem \ref{teodeoper}, one has 
\begin{align*}
e^H_{H,h}\cdot Res^G_H e^G_{H,h} \neq 0.
\end{align*}
This is a contradiction, then $H=G$  i.e. $ e_{H,h}^G=e_{G,g}^G$.   Now, we will consider the Frattini  group, $\Phi(G)\unlhd G$.  By Lemma \ref{def-frattini-fiber}
\begin{align*}
Def^G_{G/\Phi(G)} e^G_{G,g}=\frac{|O_n(G)|}{|N_G(G,g)|} \cdot |G/\Phi(G)|\cdot e^{G/\Phi(G)}_{G/\Phi(G),g\Phi(G)} \neq 0,
\end{align*}
then $\Phi(G)= \lbrace 1 \rbrace$, thus $G$ is an elementary group. If $g\neq 1$, one has 
\begin{align*}
Def_{G/\langle g \rangle}^G  e^G_{G,g}=  \frac{1}{p} e^{G/\langle g \rangle}_{G/\langle g \rangle,g\langle g \rangle} \neq 0
\end{align*}
in consequence $g=1$, and 
\begin{align*}
Def^G_{G/N} e^G_{G,1} = \frac{1-p^{r-1}}{p}  e^{G/N}_{G/N,N} =0.
\end{align*}
This means $p^{r-1} = 1 $ (mod $ q$). On the other hand, suppose $G$ is an elementary abelian  group  and  $p^{r-1} = 1 $ (mod $ q$). Then by 	Proposition \ref{res-ind-fiber} and Proposition \ref{m H g},
$e^G_{G,1} \in \underline{E_G} \cap  \underline{\underline{E_G}}$.
\end{proof}

Let $K$ and $H$ elementary abelian  subgroups of $G$ of  orders  $p^{r_K}$  and  $p^{r_H}$ respectively   with $r_K, r_H \in \mathcal{I}$. By Corollary  13 of \cite{cocskun}  the ideal generate by  $e^K_{K,1}$ is contained in the ideal generate by $e^H_{H,1}$ if only if $|K| \leq |H|$. This means $(K, e^K_{K,1}) \gg (H, e^H_{H,1}) $ if only if $|K| \leq |H|$. As a consequence, the classification given in  Theorem  \ref{thm clasific} is equal to a classification given in Corollary 13 of \cite{cocskun}.

\subsection{The slice Burnside functor}

The third example  we will see is the slice Burnside functor, this functor was  studied by  Bouc  in \cite{boucslice} and  Tounkara  in \cite{slicefunctor}. We defined  this functor in  subsection \ref{slice functor}.  In  Corollary 4.7 of \cite{boucslice} Serge Bouc   proved that  $\mathbb{Q} \otimes_\mathbb{Z }\Xi(G)$ (Definition \ref{xi(G)}) is a   finite dimensional split semisimple commutative   $\mathbb{Q}$-algebra, for all finite groups $G$.\\
By  Theorem 4.6 of \cite{boucslice},  the group $\Xi(G)$ is free with basis \-$\lbrace\langle T, S \rangle_G \in \Xi (G) | (T, S) \in 
[\Pi(G)]\rbrace$, where $[\Pi(G)]$ is a set of representatives of conjugacy classes of slices of $G$.\\
By Corollary 4.7   in  \cite{boucslice} The commutative $\mathbb{Q}$-algebra $\mathbb{Q}\Xi (G)$ (= $\mathbb{Q} \otimes_\mathbb{Z} \Xi (G)$) is a finite dimensional split semisimple commutative algebra.

\begin{nota}
For a slice $(T, S)$ of the group $G$, set
\begin{align*}
\xi_{T,S}^G := \frac{1}{|N_G(T,S)|} \sum_{U\leq S\leq V\leq T} |U| \mu (U,S)\mu (V,T) \langle U,V\rangle_G
\end{align*}
where $\mu$ is the Möbius function of the poset of subgroups of $G$ and $N_G(T, S) = N_G(T)\cap 
N_G(S)$.
\end{nota}
\begin{prop} [Theorem 5.2 of \cite{boucslice}] Let $G$ be a finite group. Then the elements $\xi^G_{
T,S}$, for $(T, S) \in [\Pi(G)]$ are the primitive idempotents of $\mathbb{Q} \Xi(G)$.
\end{prop}
\begin{prop} [Proposition 4.4 of \cite{slicefunctor}]  Let $N$ be a normal subgroup of $G$. Then 
\begin{align*}
Def^G_{G/N} \xi_{G,S}^G = m_{G,S,N} \xi_{G/N, SN/N}^{G/N},
\end{align*}
where
\begin{align*}
m_{G,S,N}=\frac{[N_G(SN): SN]}{|N_G(S)|} \sum_{\substack{ U\leq S\leq V\leq T \\ VN=G\\ UN=SN}} |U|\mu(U,S)\mu(V,T).
\end{align*}
\end{prop}
\begin{defi}[Definition 7.5. of \cite{slicefunctor}]
 A slice $(T, S)$ of $T$ is called a $T$-slice (over $\mathbb{K}$) if for any non-trivial normal
subgroup $N$ of $T$, the constant $ m_{T,S,N}$ is equal to zero.
\end{defi}
\begin{prop} \label{Tslice-MCgroup}
Let $G$ be a finite group. $G$ is an $MC$-group of $ \mathbb{Q} \Xi$  if only if there exists a $T$-slice $(G,S)$ of $G$.

\end{prop}

\begin{proof}First, suppose that $G$ is an $MC$-group of $ \mathbb{Q} \Xi$. Then there exists $\xi_{T, S}^G \in \mathbb{Q}\Xi (G)$ such that\- $\mathbb{Q}\Xi (Res^G_H) (\xi_{T, S}^G) =0$ for all $H\lneq G$, and $\mathbb{Q}\Xi(Def^G_{G/N})(\xi_{T, S}^G)=0$ for all $1\neq N \unlhd G $. If $T\neq G$ by Proposition 4.1 of \cite{slicefunctor}, we have
\begin{align*}
\mathbb{Q}\Xi(Res^G_T)(\xi_{T, S}^G)= \sum_{\substack{(T, S^\prime) \in [\Pi(T)]\\ (T,S^\prime)=_G (T, S)}} \xi^T_{(T,S^\prime)} \neq 0.
\end{align*} 
This contradiction shows that $T=G$.  Since $\mathbb{Q}\Xi (Def^G_{G/N}) (\xi_{G, S}^G) \neq 0$  for all normal subgroup $N$ of $G$ and the Proposition 4.4  in \cite{slicefunctor},  one has $(G,S)$ is a $T$-slice.  On the other hand, now suppose that there exists a $T$-slice $(G,S)$ of $G$. By Proposition 4.4 of \cite{slicefunctor}, we have
\begin{align*}
\mathbb{Q}\Xi (Def^G_{G/N})(\xi_{G,S}^G)=0
\end{align*}
for all $1\neq N\unlhd G$  and by Proposition 4.1 of \cite{slicefunctor}, one has 
\begin{align*}
\mathbb{Q}\Xi (Res^G_H) (\xi_{G,S}^G)=0
\end{align*}
for $H\subsetneqq G$. Thus $G$ is an $MC$-group of $ \mathbb{Q} \Xi$.
\end{proof}

\begin{prop} \label{cocient-relation}
Let $(T,S)$ and $(V,U)$  be slices. Then $(V,U) \twoheadrightarrow (T,S)$ (see Subsection \ref{slice functor}) if only if $(\xi_{V,U}^V, V) \gg(\xi_{T,S}^T, T) $(Definition \ref{gg}).
\end{prop}
\begin{proof}
By Theorem 7.1 of \cite{slicefunctor}, we have $(V,U) \twoheadrightarrow (T,S)$ if only if $\langle \xi_{V,U}^V  \rangle \subseteq \langle \xi_{T,S}^T \rangle$, and by Lemma \ref{lema menor},  $\langle \xi_{V,U}^V  \rangle \subseteq \langle \xi_{T,S}^T \rangle $ if only if $(\xi_{V,U}^V, V) \gg(\xi_{T,S}^T, T) $.
\end{proof}
Let  \textbf{T}-slice  the set of all $T$-slice of $\mathbb{Q}\Xi$, we denote  [\textbf{T}-slice] a set of representative of isomorphism  classes of  \textbf{T}-slice,  we have ([\textbf{T}-slice], $  \twoheadrightarrow$), is a poset. \\
 By Proposition\ref{Tslice-MCgroup},  Proposition   \ref{cocient-relation}, Proposition 8.8 in  \cite{slicefunctor} and Theorem \ref{thm clasific}. One has the poset of ideals of $\mathbb{Q}\Xi$ is isomorphic to the poset of closed subsets of [\textbf{T}-slice].

\subsection{The functor $\mathbb{K}B_K$ } %ejemplo de cuando no coinciden 

Now we will work on the shifted  Burnside functor, $\mathbb{K}B_K$  where $K$ is a finite group and $\mathbb{K}$ is  a field with characteristic 0. This example is studied in \cite{relativeBgroup} by  Serge Bouc, we defined this functor in  subsection  \ref{slice functor}.
\begin{defi}[Definition 3.1 of \cite{relativeBgroup}]
$ $
\begin{itemize}
\item For a finite group $K$, let $grp_{\Downarrow K}$ denote the following category:
\begin{itemize}
\item The objects are finite groups over $K$, i.e. pairs $(L, \phi)$, where $L$ is
a finite group and $\phi : L \longrightarrow  K $ is a group homomorphism.
\item A morphism $f : (L,   \phi) \longrightarrow  (L^\prime , \phi^\prime) $ of groups over $K$ in the category $grp_{\Downarrow K}$ is a group homomorphism $f : L \longrightarrow  L^\prime$ such that there exists some inner automorphism $i$ of $K$ with $i \circ \phi = \phi^\prime \circ f$.
\item  The composition of morphisms in $grp_{\Downarrow K}$ is the composition of
group  homomorphism, and the identity morphism of $(L, \phi)$ is the
identity automorphism of $L$.
\end{itemize}

\item If $(L,  \phi)$ and $(L^\prime , \phi^\prime)$ are groups over $K$, we say that ($L^\prime ,\phi^\prime)$ is a quotient of $(L, \phi )$, and we note $(L, \phi) \twoheadrightarrow (L^\prime , \phi^\prime)$, if there exists a morphism $f \in  Hom_{grp_{\Downarrow K}} ((L , \phi) ,  (L^\prime , \phi^\prime))$ with $f : L \longrightarrow L^\prime$ surjective. In this case, we will say that $f$ is a surjective morphism from $(L,  \phi)$ to $(L^\prime , \phi^\prime)$.
\end{itemize}
 \end{defi}
 
 \begin{nota}
When $(L, \phi)$ is a group over $K$, we denote by $L_\phi$ the
subgroup of $L \times K$ defined by 
\begin{align*}
L_\phi := \lbrace (l,\phi(l) ) \mid l\in L \rbrace.
\end{align*}
 \end{nota}

\begin{nota}
Let $(L, \phi)$ be a group over $K$. We denote by $\textbf{e}_{L,\phi}$ the ideal
of $\mathbb{K}B_K$ generated by $e_{L_\phi}^{ L\times K} \in \mathbb{K} B_K(L)$.
\end{nota}

\begin{cor}[Corollary 3.5 of \cite{relativeBgroup}] \label{p2 ideal}
Let $G$ be a finite group, and $L$ be a subgroup of $G \times K$.
Then the ideal of $\mathbb{K}B_K$ generated by  $e^{G\times K}_{L}$ is equal to the ideal of $\mathbb{K}B_K$ generated by $e^{ L\times K}_{L_{p_2}}$.
\end{cor}

\begin{defi}[Definition 4.3. of \cite{relativeBgroup}]
Let $(L, \phi)$ be a group over $K$. We say that $(L, \phi)$ is
a $B_K$-group, or a $B$-group relative to $K$, if $m_{L,N} = 0$ for every non-trivial
normal subgroup $N$ of $L$ contained in $Ker \phi$.
\end{defi}

\begin{nota}
 Let $(L, \varphi)$ be a group over $K$. If $Q$ is a normal subgroup
of $L$, contained in $Ker \varphi$, and maximal such that $m_{L,Q} \neq 0$, we denote by
$\beta_K(L, \varphi)$ the quotient $(L/Q, \varphi/Q)$ of $(L, \varphi)$.
\end{nota}

\begin{cor}[Corollary 4.9 of \cite{relativeBgroup}]\label{bk ideal}
Let $(L, \varphi)$ be a group over $K$.
\begin{enumerate}
\item $\beta_K(L, \varphi)$ is well-defined up to isomorphism in $grp_{\Downarrow K}$.
\item  $\beta_K(L, \varphi)$ is a $B_K$-group, quotient of  $(L, \varphi)$.
\item If $(P, \psi)$ is a $B_K$-group, quotient of $(L, \varphi)$, then $(P, \psi)$ is a quotient of
$\beta_K(L, \varphi)$.
\item $\textbf{e}_{L,\varphi} = \textbf{e}_{\beta_K(L,\varphi)}.$
\end{enumerate}
\end{cor}

\begin{lem}[Lemma 5.5 of \cite{relativeBgroup}] \label{bk iso}
Let $(L, \varphi)$ be a $B_K$-group, and $(M, \psi)$ be a group over $K$.
Then, $\emph{\textbf{e}}_{M,\psi} \subseteq \emph{\textbf{e}}_{L,\varphi}$ if and only if $(M, \psi) \twoheadrightarrow  (L, \varphi)$.
\end{lem}

\begin{nota}
We fix a set $\mathcal{S}_K$  of representatives of isomorphism classes
of objects in the category $grp_{\Downarrow K}$.  We let $\mathcal{B}_K$-gr denote the subset of $\mathcal{S}_K$ consisting of $B_K$-groups.

\end{nota}

\begin{thm}[Theorem 5.7 of \cite{relativeBgroup}]
 Let $\mathcal{I}_{\mathbb{K}B_K}$ be the lattice of ideals of $\mathbb{K}B_K$, ordered by inclusion of ideals, and $\mathcal{C}l_{\mathcal{B}_K\text{-}gr}$ be the lattice of closed subsets of $\mathcal{B}_K$-gr, ordered by inclusion of subsets. Then the map
\begin{align*}
I \in  \mathcal{I}_{\mathbb{K}B_K} \longmapsto \mathcal{P}_I: = \lbrace (L, \varphi) \in \mathcal{B}_K\text{-gr} \mid \textbf{e}_{(L,\varphi} \subseteq I \rbrace
\end{align*}
is an isomorphism of lattices from $\mathcal{I}_{\mathbb{K}B_K}$
to $\mathcal{C}l_{\mathcal{B}_K\text{-}gr}$. The inverse isomorphism
is the map
\begin{align*}
\mathcal{P} \in \mathcal{C}l_{\mathcal{B}_K\text{-}gr} \longmapsto I_\mathcal{P} = \sum_{(L,\varphi) \in \mathcal{P}} \textbf{e}_{(L,\varphi)},
\end{align*}
in particular $\mathcal{I}_{\mathbb{K}B_K}$ is completely distributive.
\end{thm}

\begin{prop}\label{B_K-A}
Let $L$ be a finite group. Then $L$ is an $MC$-group of $ \mathbb{K} B_K$ if and only if there exists some $X\leq L\times K$ such that $p_1(X)=L$, $k_1(X)\cap N \neq 1$ for all $1\neq N \unlhd L$ and $(X,p_2)$ is a $B_K$-group. 
\end{prop}

\begin{proof}
Suppose that $L$ is an $MC$-group for  $ \mathbb{K} B_K$, then there exists $e^{L\times K}_X \in \mathbb{K}B_K(L)$ such that
\begin{align*}
\mathbb{K}B_K(Res^L_H)(e^{L\times K}_X)=&0\\
\mathbb{K}B_K(Def^L_{L/N})(e^{L\times K}_X)=&0
\end{align*}
for all $H \lneq L$ and $1\neq N \unlhd L$. If $p_1(X)\neq L$, by Theorem 5.2.4 \cite{serge-biset}  one has 
\begin{align*}
\mathbb{K}B_K(Res^L_{p_1(X)})(e^{L\times K}_X)=\mathbb{K} B(Res^{L\times K}_{p_1(X)\times K})(e^{L\times K}_X)\neq 0
\end{align*}
this  contradiction show that $p_1(X)=L$. Moreover, if there exists $1\neq N\unlhd L$ be such that  $k_1(X)\cap N =1$, then  by Lemma 2.2 of  \cite{relativeBgroup} there exist $\lambda \neq 0$ in $\mathbb{K}$, such that
\begin{equation}\label{defBgroup}
\mathbb{K}B_K(Def_{L/N}^L)(e^{L\times K}_X)=\lambda m_{X, X\cap(N\times 1)} e^{(L/N) \times K}_{\overline{X}},
\end{equation}
where  and  $m_{X, X\cap (N\times 1)} =m_{X,1}=1$ since $X\cap (N\times 1) = (k_1(X)\cap N) \times 1$. Thus  $\mathbb{K}B_K(Def_{L/N}^G)(e^{L\times K}_X)\neq 0$, this contradiction shows that  $k_1(X)\cap N \neq 1$ for all $1\neq N \unlhd L$. Lastly, let $1\neq N \unlhd X$ 	be such that $N\leq Ker(p_2)=k_1(X)\times 1$, then $N=M\times 1$, for some $1\neq M\unlhd L$ and 
\begin{align*}
 m_{X_{p_2}, N\times 1} = m_{X,N}=m_{X,M\times 1}
\end{align*}
the first equality is by Definition 2.1 of \cite{serge-biset} and by equation  \ref{defBgroup}), one has $m_{X,M\times 1}=0$. Then $(X,p_2)$ is a $B_K$-group.\\
Conversely,  suppose that there exists some subgroup  $X$ of  $L\times K$ such that $p_1(X)=L$, $k_1(X) \cap N \neq 1$ for all $1\neq N\unlhd L$ and $(X,p_2)$ is a $B_K$-group.  By Theorem 5.2.4 of \cite{serge-biset}, one has 
\begin{align*}
\mathbb{K}B_K(Res^L_H)(e^{L\times H}_X )=0
\end{align*}
for all $H$ proper subgroup of $L$. By Lemma 2.2 of \cite{relativeBgroup},
\begin{align*}
\mathbb{K}B_K(Def^L_{L/N})(e^{L\times H}_X)= \lambda m_{X, X\cap (N\times 1)} e^{L/N \times K}_{\overline{X}},
\end{align*}
note that $X\cap (N\times 1)=(k_1(X)\cap N)\times 1$. Moreover,  
\begin{align*}
m_{X, X\cap (N\times 1)}=m_{X, (k_1(X)\cap N)\times 1}=m_{X_{p_2}, (N\times 1)\times 1} =0
\end{align*}
since $(X, p_2)$ is a $B_K$-group.  Then $ \mathbb{K}B_K(Def^L_{L/N})(e^{L\times H}_X)=0$. Thus 
$L$ is an $MC$-group  of $ \mathbb{K} B_K$.

\end{proof}

\begin{cor}
Let $L$ be a finite group. If $(L\times K, p_2)$ is a $B_K$-group, where $p_2$ is the second projection, then $L$ is an $MC$-group of $ \mathbb{K} B_K$.
\end{cor}

\begin{prop}
Let  $(e^{L\times K}_X, L)$ and  $(e^{H\times K}_Y,  H)$ be elements of $\mathcal{M}$, such that $p_1(X)=L$, $k_1(X)\cap N \neq 1$ for all $1\neq N \unlhd L$ and  $p_1(Y)=H$, $k_1(Y)\cap M \neq 1$ for all $1\neq M \unlhd H$.     Then   \emph{\textbf{e}}$_{Y_{p_2}}\gg$\emph{ \textbf{e}}$_{X_{p_2}}$ if only if $(Y,p_2) \twoheadrightarrow (X,p_2)$.
\end{prop}

\begin{proof}
One has   $e^{L\times K}_X \gg e^{H\times K}_Y $ if only if \textbf{e}$_{Y, p_2} \subseteq$ \textbf{e}$_{X, p_2}$. On the other hand, by  Theorem 5.3 of \cite{relativeBgroup},  \textbf{e}$_{Y, p_2} \subseteq$ \textbf{e}$_{X, p_2}$ if only if $$(Y,p_2) \twoheadrightarrow (X,p_2).$$
\end{proof}

\begin{thm}
Let $\mathcal{B}_K$-gr denote the subset of $\mathcal{S}_K$  consisting of $B_K$-groups, ordered by $ \twoheadrightarrow $ and let the poset $([\mathcal{M}], \gg)$  as the Definition \ref{gg} of $\mathbb{K}B_K$. Then the map
\begin{align*}
\rho: [\mathcal{M}]&\longrightarrow \mathcal{B}_K\text{-}gr\\
(e^{H\times K}_T, H) &\longmapsto \beta_k(T, p_2)
\end{align*}
is an order isomorphism.
\end{thm}
\begin{proof}
First we prove that the map  $\rho$ is well defined: Let $(e^{H\times K}_T, H)$ and $(e^{L\times K}_J, L)$ in $\mathcal{M}$ be such that $(e^{H\times K}_T, H)\sim (e^{L\times K}_J, L)$. Then the ideal generated  by $ e^{H\times K}_T $,   denoted by $\textbf{e}^{\textbf{H}\times \textbf{K}}_\textbf{T} $, is equal to the ideal generated by  $ e^{L\times K}_J $,  denoted by $\textbf{e}^{\textbf{L}\times \textbf{K}}_\textbf{J}$.
By Corollary \ref{bk ideal}, one has 
\begin{align*}
\textbf{ e}_{\beta_K (T,p2)}  = \textbf{e}^{\textbf{H}\times \textbf{K}}_\textbf{T} =\textbf{e}^{\textbf{L}\times \textbf{K}}_\textbf{J}=  \textbf{e}_{\beta_K(J,p_2)} , 
\end{align*}
then,  by lemma \ref{bk iso} $\beta_K (T, p_2) \cong \beta_K (J, p_2) $. Thus $\rho$ is well defined.\\
Now we prove that $\rho$ is an order morphism:  Let $(e^{H\times K}_T, H)$ and $(e^{L\times K}_J, L)$ in $[\mathcal{M}]$ be such that $(e^{H\times K}_T, H)\gg  (e^{L\times K}_J, L)$, this means that
$$\textbf{e}^{\textbf{H}\times \textbf{K}}_\textbf{T} \subseteq \textbf{e}^{\textbf{L}\times \textbf{K}}_\textbf{J}. $$
By Corollary \ref{bk ideal} and Lemma \ref{bk iso} one has   $\beta_K (T, p_2) \twoheadrightarrow  \beta_K (J, p_2) $. \\
Hence, all we need to show $\rho$ is  bijective:
\begin{itemize}
\item Injective.  Let $(e^{H\times K}_T, H)$ and $(e^{L\times K}_J, L)$ in $\mathcal{M}$ be such that $ \rho((e^{H\times K}_T, H))=\rho((e^{L\times K}_J, L))$. This means that
$ \beta_K (T, p_2) \cong  \beta_K (J, p_2) $. By Corollary \ref{bk ideal} and Lemma \ref{bk iso} one has  
$$\textbf{e}^{\textbf{H}\times \textbf{K}}_\textbf{T} =\textbf{e}^{\textbf{L}\times \textbf{K}}_\textbf{J}. $$
Thus $(e^{H\times K}_T, H) =(e^{L\times K}_J, L)$ i.e. $\rho$ is injective.
\item Surjective. Let $(T,\varphi) \in B_k$-gr. Then by Lemma \ref{Ideal-Agrupo} there exists an $MC$-group $H$ and $e^{H\times K}_J \in \underline{E_H} \cap \underline{\underline{E_H}}$ such that
\begin{align*}
\textbf{e}^{\textbf{H}\times \textbf{K}}_\textbf{J}=\textbf{e}^{\textbf{H}\times \textbf{K}}_{\textbf{T}_\varphi}.
\end{align*}
By Corollary \ref{p2 ideal} and Corollary \ref{bk ideal},
\begin{align*}
\textbf{ e}_{\beta_k(J, p_2)} =  \textbf{e}^{\textbf{J}\times\textbf{ K}}_{\textbf{J}_{p_2}} = \textbf{e}^{\textbf{H}\times\textbf{ K}}_\textbf{J} = \textbf{e}^{\textbf{T} \times\textbf{ K}}_{\textbf{T}_\varphi} ,
\end{align*}
then $\beta_k(J, p_2)\cong (T, \varphi)$. Thus $\rho((e^{H\times K}_J, H)) =(T,\varphi)$ i.e. $\rho$ is surjective.

\end{itemize}
 \end{proof}

\section*{Acknowledgments}
I would like to thank my advisor, Nadia Romero whose help and   wonderful suggestions  helped me  all the  time. Also thanks  to  Serge Bouc  for enlightening conversations.
I also would like to thank the following institutions, for their financial support: UNAM and   CONACYT.

\bibliographystyle{plain}
\bibliography{biblio1}

\begin{thebibliography}{10}

\bibitem{barker}
Laurence Barker.
\newblock Fibred permutation sets and the idempotents and units of monomial
  burnside rings.
\newblock {\em Journal of Algebra}, 281(2):535--566, 2004.

\bibitem{fibered}
Robert Boltje and Olcay Co{\c{s}}kun.
\newblock Fibered biset functors.
\newblock {\em Advances in Mathematics}, 339:540--598, 2018.

\bibitem{boucgreen}
Serge Bouc.
\newblock {\em Green functors and G-sets}.
\newblock 1997.

\bibitem{serge-biset}
Serge Bouc.
\newblock {\em Biset functors for finite groups}.
\newblock Springer, 2010.

\bibitem{boucslice}
Serge Bouc.
\newblock The slice {Burnside} ring and the section {Burnside} ring of a finite
  group.
\newblock {\em Compositio Mathematica}, 148(3):868--906, 2012.

\bibitem{relativeBgroup}
Serge Bouc.
\newblock Relative {B}-groups.
\newblock {\em Documenta Mathematica}, 24:2431–2462, 2019.

\bibitem{center}
Serge Bouc and Nadia Romero.
\newblock The center of a {Green} biset functor.
\newblock {\em Pacific Journal of Mathematics}, 303(2):459--490, 2020.

\bibitem{cocskun}
Olcay Co{\c{s}}kun and Deniz Y{\i}lmaz.
\newblock Fibered p-biset functor structure of the fibered burnside rings.
\newblock {\em Algebras and Representation Theory}, 22(1):21--41, 2019.

\bibitem{dressRing}
Andreas Dress.
\newblock The ring of monomial representations {I}. {Structure} theory.
\newblock {\em Journal of Algebra}, 18:137--157, 1971.

\bibitem{bourbaki}
Maurice Mashaal.
\newblock {\em Bourbaki}.
\newblock American Mathematical Soc., 2006.

\bibitem{simpleRomero}
Nadia Romero.
\newblock Simple modules over {Green} biset functors.
\newblock {\em Journal of Algebra}, 367:203--221, 2012.

\bibitem{romerofibered}
Nadia Romero.
\newblock On fibred biset functors with fibres of order prime and four.
\newblock {\em Journal of Algebra}, 387:185--194, 2013.

\bibitem{thevenaz1}
Jacques Th{\'e}venaz.
\newblock G-algebras and modular representation theory.
\newblock Technical report, Clarendon Press, 1995.

\bibitem{slicefunctor}
Ibrahima Tounkara.
\newblock The ideals of the slice {Burnside} p-biset functor.
\newblock {\em Journal of Algebra}, 495:81--113, 2018.

\end{thebibliography}
\end{document}